\documentclass{amsart}
\usepackage[utf8]{inputenc}
\usepackage{amscd,amsmath,amsxtra,amsthm,amssymb,stmaryrd,xr,mathrsfs,mathtools,multirow,hyperref}
\usepackage[all]{xy}

\newtheorem{theorem}{Theorem}[section]
\newtheorem{lemma}[theorem]{Lemma}

\newtheorem{proposition}[theorem]{Proposition}
\newtheorem{corollary}[theorem]{Corollary}

\newtheorem{remark}[theorem]{Remark}

\numberwithin{equation}{section}

\theoremstyle{remark}
\newtheorem*{example}{Example}

\newcommand{\Gal}{\operatorname{Gal}}
\newcommand{\Fil}{\operatorname{Fil}}

\newcommand{\NN}{\mathbb{N}}
\newcommand{\QQ}{\mathbb{Q}}
\newcommand{\RR}{\mathbb{R}}
\newcommand{\Qp}{\mathbb{Q}_p}
\newcommand{\Zp}{\mathbb{Z}_p}
\newcommand{\ZZ}{\mathbb{Z}}

\renewcommand{\AA}{\mathbb{A}}
\newcommand{\FF}{\mathbb{F}}

\newcommand{\Frac}{\mathrm{Frac}}
\newcommand{\ord}{\mathrm{ord}}
\newcommand{\cH}{\mathcal{H}}
\newcommand{\cris}{\mathrm{cris}}
\newcommand{\Bcris}{\mathbb{B}_\cris}
\newcommand{\lb}{[\![}
\newcommand{\rb}{]\!]}

\newcommand{\vp}{\varphi}
\newcommand{\cL}{\mathcal{L}}
\newcommand{\cC}{\mathcal{C}}

\newcommand{\cO}{\mathcal{O}}
\newcommand{\cS}{\mathcal{S}}

\newcommand{\HIw}{H^1_{\mathrm{Iw}}}

\newcommand{\T}{\mathcal{T}}

\newcommand{\AQp}{\AA_{\Qp}^+}
\newcommand{\image}{\mathrm{Im}}
\newcommand{\Dcris}{\mathbb{D}_{\rm cris}}

\newcommand{\Qpn}{\QQ_{p,n}}

\newcommand{\Sel}{\mathrm{Sel}}

\newcommand{\coker}{\mathrm{coker}}

\newcommand{\Tw}{\mathrm{Tw}}
\renewcommand{\Col}{\mathrm{Col}}

\newcommand{\corank}{\mathrm{corank}}
\newcommand{\ur}{\mathrm{ur}}

\newcommand{\Qinf}{\QQ_\infty}
\newcommand{\Qpinf}{\QQ_{\infty,p}}
\newcommand{\Qvinf}{\QQ_{\infty,v}}
\newcommand{\Kinf}{K_\infty}
\newcommand{\Kvinf}{K_{\infty,v}}
\newcommand{\Kpinf}{K_{\infty,p}}
\newcommand{\vpi}{\varpi}
\newcommand{\Afs}{A_{f,s}}
\newcommand{\Ags}{A_{g,s}}
\newcommand{\Tfs}{\T_{f,s}}

\newcommand{\As}{\mathcal{A}_{f,s}}

\begin{document}

\title[Signed Selmer Groups at Nonordinary Primes]{Arithmetic properties of signed Selmer groups at non-ordinary primes}

\author[J.~Hatley]{Jeffrey Hatley}
\address[Hatley]{
Department of Mathematics\\
Union College\\
Bailey Hall 202\\
Schenectady, NY 12308\\
USA}
\email{hatleyj@union.edu}

\author[A.~Lei]{Antonio Lei}
\address[Lei]{D\'epartement de math\'ematiques et de statistique\\
Pavillon Alexandre-Vachon\\
Universit\'e Laval\\
Qu\'ebec, QC, Canada G1V 0A6}
\email{antonio.lei@mat.ulaval.ca}

\begin{abstract}
We extend many results on Selmer groups for elliptic curves and modular forms to the non-ordinary setting. More precisely, we study the signed Selmer groups defined using the machinery of Wach modules over $\Zp$-cyclotomic extensions. First, we provide a definition of residual and non-primitive Selmer groups at non-ordinary primes. This allows us to extend techniques developed by Greenberg (for $p$-ordinary elliptic curves) and Kim ($p$-supersingular elliptic curves) to show that if two $p$-non-ordinary modular forms are congruent to each other, then the Iwasawa invariants of their signed Selmer groups are related in an explicit manner. Our results have several applications. First of all, this allows us to relate the parity of the analytic ranks of such modular forms generalizing a recent result of the first-named author for $p$-supersingular elliptic curves. Second, we can prove a Kida-type formula for the signed Selmer groups generalizing results of Pollack and Weston.
\end{abstract}


\maketitle

\section{Introduction}\label{section:intro}

\subsection{Overview}\label{section:overview}

Suppose that $\bar{\rho}:G_\QQ \rightarrow \mathrm{GL}_2(\FF)$ is an absolutely irreducible Galois representation, where $\FF$ is a finite field of characteristic $p$. If $\bar{\rho}$ is modular, then we can find many modular forms $f$ such that the associated Galois representation $\rho_f: G_\QQ \rightarrow \mathrm{GL}_2(\bar{\QQ}_p)$ gives rise to $\bar{\rho}$, in the sense that its semisimple reduction mod $p$ is isomorphic to $\bar{\rho}$. For the purposes of this introduction, let us write $\cC(\bar{\rho})$ for the set of all such modular forms. Two modular forms $f,g \in \cC(\bar{\rho})$ are said to be congruent mod $p$.

Suppose that $f,g \in \cC(\bar{\rho})$ are of weight $k=2$ with rational Fourier coefficients, so that they correspond to elliptic curves, and suppose further that they are $p$-ordinary. Let $\Kinf$ be the $\Zp$-cyclotomic extension of $\QQ$. In \cite{greenbergvatsal}, Greenberg and Vatsal study the relation between the Iwasawa invariants of the Selmer groups $\Sel(f/\Kinf)$ and $\Sel(g/\Kinf)$. By defining auxiliary Selmer groups (which they call non-primitive), they are able to show that knowledge of the Iwasawa invariants for $\Sel(f/\Kinf)$ is sufficient to compute those  for $\Sel(g/\Kinf)$. The ideas and techniques in \textit{loc. cit.} were extended to more general modular forms (still $p$-ordinary) in \cite{epw}, in which Emerton, Pollack, and Weston show that Iwasawa invariants are well-behaved in Hida families.

In all of this work, the $p$-ordinary assumption is crucial, as the classical Selmer group  for $f$ at a non-ordinary prime $p$ is not $\Lambda$-cotorsion,  where $\Lambda$ is the Iwasawa algebra $\Zp\lb X\rb$. So, initially it does not even make sense to talk about Iwasawa invariants at such primes. However, Kobayashi \cite{kobayashi03} defined the so-called plus/minus Selmer groups when $f$ corresponds to an elliptic curve with $a_p=0$ and showed that they are $\Lambda$-cotorsion. Subsequently, the results of \cite{greenbergvatsal} were generalized to these plus/minus Selmer groups by Kim in \cite{kim09, kim13}.

Kobayashi's definition of  plus/minus Selmer groups  has been generalized to supersingular elliptic curves with $a_p\ne0$ by Sprung \cite{sprung09} and to general non-ordinary modular forms by Loeffler and Zerbes together with the second-named author in \cite{leiloefflerzerbes10} using the machinery of Wach modules. Indeed,   it  is possible to define two \textit{signed} Selmer groups over $\Kinf$ attached to $f$, which we denote  by $\Sel_i(f/\Kinf)$ in the introduction.  These Selmer groups have been proved  to be $\Lambda$-cotorsion in a large number of cases. Accordingly, at a non-ordinary prime, $f$ has \textit{pairs} of Iwasawa invariants. We shall review the definition of $\Sel_i(f/\Kinf)$ in \S\ref{section:signedselmer}. Along the way, we shall also define the counterparts  over $\QQ$ and the non-primitive version of these groups, which are denoted by $\Sel_i(f/\QQ)$ and $\Sel_i^{\Sigma_0}(f/\Kinf)$ respectively.

The goal of the present paper is to extend the results of Greenberg-Vatsal, Emerton-Pollack-Weston and Kim to modular forms of even weight $k \geq 2$ at non-ordinary primes. Below is a summary of our main results.
\begin{itemize}
\item The signed Selmer groups have no proper $\Lambda$-submodules of finite index (Theorem~\ref{thm:finiteindex}).
\item If $f$ and $g$ are two forms of the same weight inside $\cC(\bar\rho)$,  the $\mu$-invariant of $\Sel_i(f/\Kinf)$ vanishes if and only if that of $\Sel_i(g/\Kinf)$ does. In this case, the $\lambda$-invariants of $\Sel_i^{\Sigma_0}(f/\Kinf)$ and $\Sel_i^{\Sigma_0}(g/\Kinf)$ coincide (Theorem~\ref{thm:sigma-invariants}).
\item Let $f$ and $g$ be as above. There exist some explicitly computable finite sets $\cS_f$ and $\cS_g$ such that
\[
\corank\ \Sel_i(f/\QQ)+|\cS_f|\equiv
\corank\ \Sel_i(g/\QQ)+|\cS_g|\mod2
\]
(Theorem~\ref{thm:main}). Furthermore, if $r_f$ and $r_g$ denote the analytic ranks of $f$ and $g$ respectively, and if $f$ and $g$ have trivial nebentypus, then
\[
r_{\mathrm{an}}(f) + \vert \cS_{f} \vert \equiv r_{\mathrm{an}}(g) + \vert \cS_{g} \vert \mod 2
\]
(Corollary~\ref{cor:parity}).
\item If $F/\QQ$ is a finite extension where $p$ is unramified, we may extend the definition of signed Selmer groups to $\Sel_i(f/F_\infty)$, where $F_\infty$ is the cyclotomic $\Zp$-extension of $F$ (this is the content of \S\ref{sec:urselmer}).
\item A Kida-type formula as proved by Pollack-Weston in \cite{pollackweston06} holds, relating the $\lambda$-invariants of $\Sel_i(f/F'_\infty)$ and $\Sel_i(f/F_\infty)$, where $F'/F$ is a finite Galois $p$-extension (Theorem~\ref{thm:kida}).
\item In the final section, we discuss the various hypotheses imposed in our paper and mention some computational verification of Corollary~\ref{cor:parity}.
\end{itemize}

\subsection{Notation}\label{section:notation}
Throughout this article, $p$ denotes a fixed odd prime.

For each integer $n\ge0$, we write $\Qpn=\Qp(\mu_{p^n})$. Let $\Qinf=\QQ(\mu_{p^\infty})$ and $\Gamma=\Gal(\Qinf/\QQ)=\Gal(\Qp(\mu_{p^\infty})/\Qp)$, which we decompose as $\Gamma_0\times\Delta$, where $\Gamma_0\cong\Zp$ and $\Delta\cong \ZZ/(p-1)\ZZ$. Let $\chi$ be the cyclotomic character on $\Gamma$ and let $\kappa$ and $\omega$ be the restriction of $\chi$ to $\Gamma_0$ and $\Delta$ respectively. Note that $p$ is totally ramified in $\Qinf$. If $K$ is a subfield of $\Qinf$, we shall abuse notation and write $p$ for the unique prime of $K$ lying above $p$.

Let $K_\infty$ be the $\Zp$-cyclotomic extension of $\QQ$ (so $K_\infty=\Qinf^{\Delta}$, with $\Gal(K_\infty/\QQ)\cong\Gamma_0$). For $n\ge0$, we write $K_n$ for the subextension of $K_\infty/\QQ$ of degree $p^n$.
Let $E$ be a fixed finite extension of $\Qp$ with ring of integers $\cO_E$. We write $\Lambda$ for the Iwasawa algebra $\cO_E\lb \Gamma_0\rb$, which may be identified with the power series ring $\cO_E\lb X\rb$. We also write $\Lambda(\Gamma)$ for the Iwasawa algebra of $\Gamma$ over $\cO_E\lb \Gamma\rb$, which we may identify with the power series ring $\cO_E[\Delta]\lb X\rb$.

Given a character $\theta$ on $\Delta$, we write $e_\theta=\frac{1}{p-1}\sum_{\sigma\in\Delta}\theta(\sigma)^{-1}\cdot\sigma$ for the corresponding idempotent element in $\cO[\Delta]$. If $M$ is an $\cO_E\lb \Gamma\rb$-module $M$, we shall write $M^\theta=e_\theta\cdot M$ for its $\theta$-isotypic component. When $\theta$ is the trivial character on $\Delta$, we write $M^\Delta=M^\theta$.

\subsection{Sketch of proofs}

For the reader's convenience, we now briefly outline the main ideas of the arguments which culminate in the proof of our parity result (Corollary~\ref{cor:parity}). It will be clear to the reader that we were greatly influenced by the work of Greenberg-Vatsal \cite{greenbergvatsal}.

We begin in Section \ref{S:signed} by recalling the theory of Wach modules related to $p$-adic representations, especially some of the structure theorems due to Berger. Upon choosing a basis for the Wach module associated to a modular form $f$, we construct Coleman maps
\[
\Col_{f,i}:\HIw(\Qp,\T_f)\rightarrow \Lambda(\Gamma), \quad i=1,2
\]
from the Iwasawa cohomology of $\T_f$ to the Iwasawa algebra, where $\T_f$ is a certain lattice inside Deligne's $p$-adic representation attached to $f$. For any subfield $K \subset \Qinf$, we then define our signed Selmer groups $\Sel_i(f / K)$ in terms of the kernels of these Coleman maps.

Upon relaxing the local conditions away from $p$, we also define \textit{non-primitive} signed Selmer groups $\Sel_i^{\Sigma_0}(f /K)$, and in Section \ref{section:relation-non-primitive} we study the relationship between $\Sel_i(f /K)$ and $\Sel_i^{\Sigma_0}(f /K)$. In particular, we show that their quotient is a finite product of local Galois cohomology groups which are computable in practice (see Corollary \ref{cor:selmerquotient}).

In Section \ref{section:signed-selmer-over-QQ} we study the relation between our signed Selmer groups and the standard Bloch-Kato Selmer groups, showing that our $\Sel_2(f / \QQ)$ is equal to $\Sel_{\mathrm{BK}}(f / \QQ)$ (see Proposition \ref{prop:signed-eq-bk}).

We begin proving our main results in Section \ref{section:submodulesfinite}, where we first establish the non-existence of submodules of finite index (Theorem \ref{thm:finiteindex}). This result is a crucial ingredient in the proofs of Section \ref{section:invariants}, where we prove that the non-primitive signed Selmer groups $\Sel_i^{\Sigma_0}(f / \Qinf)$ are completely determined by the residual Galois representation associated to $f$; therefore, if $f,g \in \mathcal{C}(\bar{\rho})$ then the Iwasawa invariants of their associated signed Selmer groups $\Sel_i(f / \QQ)$ and $\Sel_i(g / \QQ)$ are related by an explicit formula in terms of local Galois cohomology groups by Corollary \ref{cor:selmerquotient}. In Section \ref{section:parity}, we study the dimensions of these local Galois cohomology groups modulo $2$, and in light of Proposition \ref{prop:signed-eq-bk} (which relates our signed Selmer groups to the Bloch-Kato Selmer groups), we invoke proven cases of the Parity Conjecture to obtain our own parity result.

\section{Signed Selmer groups}\label{section:signedselmer}
\subsection{Definition of Coleman maps and signed Selmer groups}\label{S:signed}

Let  $f=\sum a_n(f)q^n$ be a normalised new cuspidal modular eigenform of even weight $k \ge 2$, level $N$ and nebentypus $\varepsilon$.  We assume that $a_n(f)$ is defined over a totally real field for all $n$ and that $f$ is non-ordinary at $p$ (that is, $a_p(f)$ is not a $p$-adic unit).


We fix a finite extension $E$ of $\Qp$ such that $a_n(f)\in E$ for all $n$ and $\vpi$ a uniformizer of $E$. We let $V_f$ be the $E$-linear Galois representation attached to $f$ constructed by Deligne \cite{deligne69}. Then $V_f$ has Hodge--Tate weights $\{0, 1-k\}$, where our convention  is that the Hodge--Tate weight of the cyclotomic character is $1$.  Let $T_f$ be the canonical $G_{\QQ}$-stable $\cO_E$-lattice in $V_f$ defined by Kato \cite[\S8.3]{kato04}. Let $A_f=V_f/T_f(1)$ and $\T_f$  the Tate twist $T_f(k-1)$, which has Hodge--Tate weights $\{0,k-1\}$. There is a perfect pairing
\[
A_f\times \T_{f^*}\rightarrow\mu_{p^\infty},
\]
where $f^*$ is the modular form whose Fourier coefficients are given by the complex conjugation of those of $f$ (see for example \cite[\S14.10]{kato04}). Since we assume that $a_n(f)\in\RR$ for all $n$, we have in fact a perfect pairing
\[
A_f\times \T_{f}\rightarrow\mu_{p^\infty}.
\]

Let $\AQp=\Zp\lb\pi\rb$, which is equipped with the Frobenius map $\vp(\pi)=(1+\pi)^p-1$ and an action by $\Gamma$ given by $\gamma\cdot \pi=(1+\pi)^{\chi(\gamma)}-1$ for $\gamma\in\Gamma$.
For a free rank-$d$ $\cO_E$-module $T$ equipped with a continuous action by $G_{\Qp}$, we write $\NN(T)$ and $\Dcris(T)$ for the Wach module and the Dieudonn\'e module of $T$ (c.f. \cite{berger03,berger04}).
We write
\[
\HIw(\Qp,T)=\varprojlim H^1(\Qpn,T)
\]
where the connecting maps are corestrictions.

 Recall the canonical $\Lambda(\Gamma)$-isomorphism
\begin{equation}\label{eq:Herr}
\HIw(\Qp,\T_f)\cong \NN(\T_f)^{\psi=1}
\end{equation}
proved by Berger \cite{berger03}, where $\psi$ is a left-inverse of $\vp$ as described in \cite[\S I.2]{berger03}. By an abuse of notation, we shall identify the two modules with each other. The Wach module $\NN(\T_f)$ is free of rank $2$ over $\cO_E\otimes\AQp$. We fix a basis $n_1$, $n_2$.
Via the $\Lambda(\Gamma)$-homomorphism
\[
1-\vp:\NN(\T_f)^{\psi=1}\rightarrow \vp^*\NN(\T_f)^{\psi=0}
\]
and the fact that $\vp^*\NN(\T_f)^{\psi=0}$ is a free $\Lambda(\Gamma)$-module of rank $2$ generated by the elements $(1+\pi)\vp(n_1)$ and $(1+\pi)\vp(n_2)$, we define the Coleman maps
\[
\Col_{f,i}:\HIw(\Qp,\T_f)\rightarrow \Lambda(\Gamma)
\]
for $i=1,2$, given by the relation
\[
(1-\vp)(z)=\Col_{f,1}(z)\cdot (1+\pi)\vp(n_1)+\Col_{f,2}(z)\cdot (1+\pi)\vp(n_2).
\]

If $T$ is an $\cO_E$-linear representation of $G_{\Qp}$, we write
$$\cL_T:\HIw(\Qp,T)\rightarrow \Frac(\cH_E(\Gamma))\otimes \Dcris(T)$$
for Perrin-Riou's big logarithm map, where $\cH_E(\Gamma)$ is the algebra of $E$-valued distributions on $\Gamma$. When the Hodge-Tate weights of $T$ are non-negative (e.g. $\T_f$), this map lands inside $\cH_E(\Gamma)\otimes \Dcris(T)$.

For every integer $m$, we have the Tate twist $T(m)=T\cdot e_m$, where  $e_m$ is a basis on which $\Gamma$ acts via $\chi^m$. Recall that there is an element $t$ in Fontaine's ring $\textbf{B}_{\mathrm{dR}}$ such that $\varphi(t)=pt$ and $\gamma \cdot t = \chi(\gamma)t$ for all $\gamma \in \Gamma$. We have the natural maps
\[
\HIw(\Qp,T(-m))\stackrel{e_{m}}{\longrightarrow} \HIw(\Qp,T),\quad \Dcris(T)\xrightarrow{e_{-m}t^{m}} \Dcris(T(-m))
\]
and
\begin{equation}\label{eq:twist}
\cL_{T(-m)}(z)=(\ell_{-1}\cdots \ell_{-m})^{-1}\cdot\Tw_{-m}\circ \cL_T(z\cdot e_m)\otimes e_{-m}t^{m}
\end{equation}
for all $z\in \HIw(\Qp,T(-m))$. Here, $\ell_i=\log(\gamma)/\log(\chi(\gamma))-i$ and $\Tw_{-m}$ denotes the $E$-linear map that sends $\sigma\in\Gamma$ to $\chi^m(\sigma)\sigma$.

Let $\nu_1,\nu_2$ be the basis of $\Dcris(\T_f)$ obtained from $n_1,n_2$ modulo $\pi$. The Coleman maps defined above give rise to the decomposition
\begin{equation}\label{eq:decomp}
\cL_{\T_f}=\begin{pmatrix}
\nu_1&\nu_2
\end{pmatrix}
\cdot M\cdot
\begin{pmatrix}
\Col_{f,1}\\ \Col_{f,2}
\end{pmatrix}
\end{equation}
for some $2\times 2$ logarithmic matrix $M$ that is defined over $\cH_E(\Gamma)$ (see \cite[\textsection 3A]{leiloefflerzerbes11}). We shall choose $n_1,n_2$ so that $\nu_1$ is an $\cO_E$-basis of $\Fil^0\Dcris(\T_f)$ and $\vp(\nu_1)=\nu_2$ \cite[Lemma 3.1]{leiloefflerzerbes15}.
Let $v_i=\nu_i\otimes e_{2-k}t^{k-2}\in\Dcris(T_f(1))$. Then $v_1, v_2$ is a basis of $\Dcris(T_f(1))$. Under the crystalline pairing
\[
\langle\sim,\sim\rangle_{\cris}:\Dcris(\T_f)\times \Dcris(T_f(1))\rightarrow\Dcris(\cO_E(1))\cong \cO_E,
\]
we have
\begin{equation}\label{eq:pairbasis}
\langle\nu_i,v_i\rangle_\cris=0,\ i=1,2,\quad \langle\nu_1,v_2\rangle_\cris=-\langle\nu_2,v_1\rangle_\cris
\end{equation}
(see \cite[p.494]{leiloefflerzerbes10}).

The following is a generalization of \cite[Lemma~3.2]{leiponsinet}.

\begin{lemma}\label{lem:orth}
Let
\begin{align*}
\langle\sim,\sim\rangle:\HIw(\Qp,\T_f)\times\HIw(\Qp,T_f(1))&\rightarrow \Lambda(\Gamma)\\
((x_n),(y_n))&\mapsto\varprojlim \sum_{\sigma\in\Gamma/\Gamma^{p^n}}\langle x_n,y_n^\sigma\rangle_n,
\end{align*}
where $\langle\sim,\sim\rangle_n$ is the Tate pairing
\[
H^1(\Qpn,\T_f)\times H^1(\Qpn,T_f(1))\rightarrow \cO_E.
\]
Under this pairing, $\ker\Col_{f,i}\subset \HIw(\Qp,\T_f)$ is the orthogonal complement of its natural image in $\HIw(\Qp,T_f(1))$ under the natural map
$$
\HIw(\Qp,\T_f)\stackrel{e_{2-k}}\rightarrow\HIw(\Qp,T_f(1)).
$$
\end{lemma}
\begin{proof}
Let $x\in \HIw(\Qp,\T_f)$ and $y\in\HIw(\Qp,T_f(1))$.
By \cite[Lemme~3.6.1(ii)]{perrinriou94}, $\langle x,y\rangle=0$ if and only if $\langle x\cdot e_{(2-k)/2},y\cdot e_{(k-2)/2}\rangle =0$.
The explicit reciprocity law of Perrin-Riou from \cite[\S3.6.4]{perrinriou94} (as proved in \cite[Theorem~B.6]{loefflerzerbes14}) tells us that this is equivalent to
\[
\langle \cL_{T_f(k/2)}(x\cdot e_{(2-k)/2}),\cL_{T_f(k/2)}(y\cdot e_{(k-2)/2})\rangle_\cris=0,
\]
where the pairing $\langle\sim,\sim\rangle_\cris$ is extended $\Frac(\cH_E(\Gamma))$-linearly.
From \eqref{eq:twist}, \eqref{eq:decomp}, and \eqref{eq:pairbasis}, we deduce that this is the same as saying
\[
\Tw_{(2-k)/2}\det\left(M\cdot\begin{pmatrix}
\Col_{f,1}(x)&\Col_{f,1}(y\cdot e_{k-2})\\
\Col_{f,2}(x)&\Col_{f,2}(y\cdot e_{k-2})
\end{pmatrix}\right)=0.
\]
This shows that $\langle x,y\rangle =0$ for all $x\in \ker\Col_{f,i}$ if and only if $y\cdot e_{k-2}\in\ker\Col_{f,i}$. Hence the result.
\end{proof}

By Tate duality, we have the pairing
\begin{equation}\label{eq:finitepairing}
H^1(K,\T_f)\times H^1(K,A_f)\rightarrow \Qp/\Zp
\end{equation}
for all finite extensions $K/\Qp$ contained inside $\Qpinf$. 


For the rest of the paper, we make the following two hypotheses.

\begin{itemize}
\item[\textbf{(irred)}] The $G_\QQ$-representation $T_f/\vpi T_f$ is irreducible.

\item[\textbf{(inv)}] For all $m \geq 0$,  $A_f(m)^{G_{\Qpinf}}=0$.
\end{itemize}
\begin{remark}
The first hypothesis is very mild, since all but finitely many of the residual $G_\QQ$-representations associated to $f$ are irreducible (c.f. Lemma~\ref{lem:irred} below). The second hypothesis also holds in many cases. For instance, if $k \leq p$ then this follows from the proof of \cite[Lemma 4.4]{lei09}.
\end{remark}

\begin{lemma}\label{lem:injsurj}
If $K/\Qp$ is a finite  Galois extension contained inside  $\Qpinf$, then
the natural projection map $\HIw(\Qp,\T_f)\rightarrow H^1(K,\T_f)$ is surjective, whereas the restriction map $H^1(K,A_f)\rightarrow H^1(\Qpinf,A_f)$ is injective.
\end{lemma}\begin{proof}
If $K\subset L\subset \Qpinf$ are Galois extensions of $\Qp$, then \textbf{(inv)} together with the inflation-restriction exact sequence tells us that the restriction map $H^1(K,A_f)\rightarrow H^1(L,A_f)$ is injective. Hence, the corestriction map $H^1(L,\T_f)\rightarrow H^1(K,\T_f)$ is surjective. Since the inverse limit is left-exact, we see that $H^1(K,A_f)\rightarrow H^1(\Qpinf,A_f)$ is injective, and by duality we have that $\HIw(\Qp,\T_f)\rightarrow H^1(K,\T_f)$ is surjective.
\end{proof}

For each finite Galois extension $K/\Qp$ that is contained inside $\Qpinf$, we define
\[
H^1_i(K,\T_f)\subset H^1(K,\T_f)
\]
to be the image of $\ker\Col_{f,i}$ under the natural map $\HIw(\Qp,\T_f)\rightarrow H^1(K,\T_f)$. Similarly, we define $H^1_i(K,T_f(1))$ to be the image of $\ker\Col_{f,i}\cdot e_{2-k}$ under the natural projection map $\HIw(\Qp,T_f(1))\rightarrow H^1(K,T_f(1))$.

We have a natural exact sequence
\begin{equation}\label{eq:K-short-exact-sequence}
0 \rightarrow H^1(K,T_f(1)) \xrightarrow{\iota_K} H^1(K,V_f(1)) \xrightarrow{\pi_K} H^1(K,A_f) \rightarrow 0,
\end{equation}
where the injectivity of $\iota_K$ follows from \textbf{(inv)}, and the surjectivity of $\pi_K$ follows by the same lemma combined with Tate duality. We define $H^1_i(K,V_f(1))$ to be the $E$-vector space generated by the image of $H^1_i(K,T_f(1))$ under $\iota_K$, and we define $H^1_i(K,A_f)$ to be the image of $H^1_i(K,V_f(1))$ under $\pi_K$. In particular, we have the isomorphism
\[
H^1_i(K,A_f)\cong H^1_i(K,V_f(1))/H^1_i(K,T_f(1))\cong H^1_i(K,A_f)\otimes E/\cO_E.
\]

\begin{lemma}\label{lem:orthfinite}
Under \eqref{eq:finitepairing}, $H^1_i(K,\T_f)$ and $H^1_i(K,A_f)$ are orthogonal complements of each other.
\end{lemma}
\begin{proof}
Let $\mathcal{Y} \subset H^1(K,A_f)$ be the orthogonal complement of $H^1_i(K,\T_f)$ under \eqref{eq:finitepairing}. Then $H^1_i(K,A_f) \subset \mathcal{Y}$ by Lemma \ref{lem:orth} and the bilinearity of the Tate pairing. The reverse containment follows from Lemma \ref{lem:orth} and the exactness of \eqref{eq:K-short-exact-sequence}.
\end{proof}

\begin{remark}\label{rem:restrictedHi}
Lemma~\ref{lem:injsurj} tells us that $\HIw(\Qp,\T_f)_{\Gal(\Qpinf/K)}=H^1(K,\T_f)$ and $H^1(\Qpinf,A_f)^{\Gal(\Qpinf/K)}=H^1(K,A_f)$. Hence, $(\ker\Col_{f,i})_{\Gal(\Qpinf/K)}=H^1_i(K,\T_f)$ and $H^1_i(\Qpinf,A_f)^{\Gal(\Qpinf/K)}=H^1_i(K,A_f)$.
\end{remark}

On taking inverse and direct limits in \eqref{eq:finitepairing}, we have the pairing
\begin{equation}\label{eq:pairing}
[\bullet,\bullet]:\HIw(\Qp,\T_f)\times H^1(\Qpinf,A_f)\rightarrow\Qp/\Zp.
\end{equation}

We define $H^1_i(\Qpinf,A_f)$ to be the direct limit $\varinjlim H^1_i(K,A_f)$. Then  $\ker\Col_{f,i}$ and $H^1_i(\Qpinf,A_f)$ are orthogonal complements of each other under \eqref{eq:pairing} by Lemma~\ref{lem:orthfinite}.

 If $v$ is a place of $\Qinf$ that does not divide $p$, we define
$$H^1_i(\Qvinf,A_f)=H^1_\ur(\Qvinf,A_f),$$ that is, the unramified subgroup of
$H^1(\Qvinf,A_f)$. We shall use the notation $H^1_{/i}(\cdots)$ to denote that quotient $H^1(\cdots)/H^1_i(\cdots)$. We  define the signed Selmer group over $\Qinf$ by
\[
\Sel_i(A_f/\Qinf)=\ker\left(H^1(\Qinf,A_f)\rightarrow\prod_v H^1_{/i}(\Qvinf,A_f)\right),
\]
where $v$ runs through all places of $\Qinf$.

If $K$ is any subfield of $\Qinf$ (e.g. $\Kinf$, $K_n$, etc.), we may define $H^1_i(K_v,A_f)\subset H^1(K_v,A_f)$ for any place $v\nmid p$ of $K$ in the same way and define the signed Selmer groups
\[
\Sel_i(A_f/K)=\ker\left(H^1(K,A_f)\rightarrow\prod_v H^1_{/i}(K_v,A_f)\right).
\]
In particular, we have the equality
\[
\Sel_i(A_f/K_\infty)=\Sel_i(A_f/\Qinf)^\Delta.
\]

Let $\Sigma$ be a set of places of $\QQ$ that contains $p$, all the primes that divide $N$ and $\infty$. We write $\QQ_\Sigma$ for the  maximal extension of $\QQ$ that is unramified outside $\Sigma$. If $K$ is a field contained inside $\QQ_\infty$, we shall write $H^i_\Sigma(K,*)$ for the Galois cohomology $H^i(\QQ_\Sigma/K,*)$. Note that we have the equality
\[
\Sel_i(A_f/K)\cong\ker\left(H^1_\Sigma(K,A_f)\rightarrow\prod_{v|\ell,\ell\in\Sigma}H^1_{/i}(K_v,A_f)\right).
\]

If $\Sigma_0$ is a subset of $\Sigma$ that does not contain $p$ and $\infty$, we define the \textit{non-primitive} signed Selmer groups
\[
\Sel_i^{\Sigma_0}(A_f/K)=\ker\left(H^1_\Sigma(K,A_f)\rightarrow\prod_{v|\ell,\ell\in\Sigma\setminus\Sigma_0}H^1_{/i}(K_v,A_f)\right).
\]

Given an integer $s$, we write $\Afs=A_f\otimes \chi^s$ and $\Tfs=\T_f\otimes \chi^{-s}$. We have the natural isomorphism
\[
H^1(\Qinf,\Afs)\cong H^1(\Qinf,A_f)\otimes\chi^s,
\]
which is compatible with
\[
H^1(\Kvinf,\Afs)\cong H^1(\Kvinf,A_f)\otimes\chi^s
\]
for all $v$. On defining $H^1_i(\Qvinf,\Afs)$
to be the image of $H^1_i(\Qvinf,A_f)$, we may equally define the signed Selmer groups for $\Afs$, namely,
\[
\Sel_i(\Afs/\Qinf)=\ker\left(H^1(\Qinf,\Afs)\rightarrow\prod_v H^1_{/i}(\Qvinf,\Afs)\right).
\]

As before, given a subfield $K$ of $\Qinf$, we may define $H^1_i(K_p,\Tfs)\subset H^1(K_p,\Tfs)$ to be the image of the projection of $\ker\Col_{f,i}\otimes\chi^{-s}$, which allows us to define $H^1_i(K_p,\Afs)$ as before. On defining $H^1_i(K_v,\Afs)$ for $v\nmid p$ similarly, we may define $\Sel_i(\Afs/K)$ as before. Given a subset $\Sigma_0$ of $\Sigma$ that does not contain $p$ and $\infty$, we may equally define the non-primitive Selmer groups $\Sel_i^{\Sigma_0}(\Afs/K)$ in the same way.

\begin{remark}\label{selmer-twists} We note that $\Sel_i(\Afs/\Qinf)$ and $\Sel_i^{\Sigma_0}(\Afs/\Qinf)$ are isomorphic to
$\Sel_i(A_f/\Qinf)\otimes\chi^s$ and $\Sel_i^{\Sigma_0}(A_f/\Qinf)\otimes\chi^s$ respectively as $\cO_E\lb \Gamma\rb$-modules.  Since $\chi^s=\kappa^s\times\omega^s$, which corresponds to the decomposition $\Gamma=\Gamma_0\times\Delta$, we have the $\Lambda$-isomorphism
\[
\Sel_i(\Afs/\Kinf)\cong\Sel_i(\Afs/\Qinf)^\Delta\cong\Sel_i(A_f/\Qinf)^{\omega^{-s}}\otimes\kappa^s,
\]
where $\Sel_i(A_f/\Qinf)^{\omega^{-s}}$ is the $\omega^{-s}$-isotypic component of $\Sel_i(A_f/\Qinf)$. More generally, if $\theta$ is any character on $\Delta$, the following isomorphism of $\Lambda$-modules holds:
\[
\Sel_i(\Afs/\Qinf)^\theta\cong\Sel_i(A_f/\Qinf)^{\theta\omega^{-s}}\otimes\kappa^s.
\]
In fact, via the isomorphism
\[
H^1(\Qinf,\Afs)^\theta\cong H^1(\Kinf,\Afs(\theta^{-1})),
\]
we may define $H^1_i(\Kvinf,\Afs(\theta^{-1}))$ for any place $v$ of $\Kinf$ and  hence $\Sel_i(\Afs(\theta^{-1})/\Kinf)$ in a natural way. This would coincide with the isotypic component $\Sel_i(\Afs/\Qinf)^\theta$.
 \end{remark}

Finally, for a subfield $K$ of $\Qinf$, we define the corresponding signed Selmer groups for $\Tfs$ over $K$ by
\[
\Sel_i(K,\Tfs):=\ker\left(H^1(K_n,\Tfs)\rightarrow \prod_{v}H^1_{/i}(K_{v},\Tfs)\right),
\]
where the product runs through all places of $K$.

\subsection{Relation between signed Selmer groups and their non-primitive counterparts}\label{section:relation-non-primitive}

Let $s\in\ZZ$ and $n\ge 0$ be integers. We recall the following Poitou-Tate exact sequences from \cite[Proposition~A.3.2]{perrinriou95}:
\begin{equation}\label{eq:PT1finite}\begin{split}
0\rightarrow H^1_i(\QQ(\mu_{p^n}),\Tfs)\rightarrow H^1_\Sigma(\QQ(\mu_{p^n}),\Tfs)\rightarrow \bigoplus H^1_{/i}(\QQ(\mu_{p^n})_{v},\Tfs)\rightarrow\\
\Sel_i(\Afs/\QQ(\mu_{p^n}))^\vee\rightarrow  H^2_\Sigma(\QQ(\mu_{p^n}),\Tfs)
,\end{split}
\end{equation}
\begin{equation}\label{eq:PT2finite}\begin{split}
0\rightarrow \Sel_i(\Afs/\QQ(\mu_{p^n}))\rightarrow H^1_\Sigma(\QQ(\mu_{p^n}),\Afs)\rightarrow \bigoplus H^1_{/i}(\QQ(\mu_{p^n}),\Afs)\rightarrow\\
 H^1_i(\QQ(\mu_{p^n}),\Tfs)^\vee
\rightarrow H^2_\Sigma(\QQ(\mu_{p^n}),\Afs)\rightarrow \bigoplus  H^2(\QQ(\mu_{p^n})_v,\Afs).\end{split}
\end{equation}
On taking inverse (resp. direct) limits, we have:
\begin{equation}\label{eq:PT1}\begin{split}
0\rightarrow \varprojlim H^1_i(\QQ(\mu_{p^n}),\Tfs)\rightarrow \varprojlim H^1_\Sigma(\QQ(\mu_{p^n}),\Tfs)\rightarrow \varprojlim\bigoplus  H^1_{/i}(\QQ(\mu_{p^n})_v,\Tfs)\rightarrow\\
\Sel_i(\Afs/\Qinf)^\vee\rightarrow  \varprojlim H^2_\Sigma(\QQ(\mu_{p^n}),\Tfs)
,\end{split}
\end{equation}
\begin{equation}\label{eq:PT2}\begin{split}
0\rightarrow \Sel_i(\Afs/\Qinf)\rightarrow H^1_\Sigma(\Qinf,\Afs)\rightarrow \bigoplus H^1_{/i}(\Qvinf,\Afs)\rightarrow\\
 \varinjlim H^1_i(\QQ(\mu_{p^n}),\T_f)^\vee
\rightarrow H^2_\Sigma(\Qinf,\Afs)\rightarrow\varinjlim \bigoplus  H^2(\QQ(\mu_{p^n})_{v},\Afs).\end{split}
\end{equation}

From now on we shall assume the following hypothesis holds.

\begin{itemize}
\item[\textbf{(tor)}] Let $\theta$ be a character on $\Delta$ and $i\in\{1,2\}$. The Selmer group $\Sel_i(A_{f,s}/\Qinf)^\theta$ is cotorsion over $\Lambda$.
\end{itemize}

This is known to hold in many cases; see for instance \cite[Theorem~7.3]{kobayashi03}, \cite[Proposition 6.4]{lei09}, \cite[Theorem~6.5]{leiloefflerzerbes10} and \cite[Theorem~7.14]{sprung09}.

\begin{lemma}\label{lem:vanishH1i}
The inverse limit $\varprojlim H^1_i(\QQ(\mu_{p^n}),\Tfs)$ is $0$.
\end{lemma}
\begin{proof}
By twisting, Remark~\ref{selmer-twists} allows us to assume that $s=0$. Furthermore, it is enough to show that the claim holds at each $\Delta$-isotypic component. Let $\theta$ be a character on $\Delta$. We shall show that $\varprojlim H^1_i(\QQ(\mu_{p^n}),\T_f)^\theta=0$

We recall from \cite[Lemma~6.2]{lei09} that if $v\nmid p$, we have
\[
\varprojlim H^1_{/i}(\QQ(\mu_{p^n})_{v},\T_f)=0.
\]
Consequently, the third term of the exact sequence \eqref{eq:PT1} simplifies to be
\[
H^1_{/i}(\Qpinf,\T_f)\cong \image\Col_{f,i}\subset \cO_E\lb\Gamma\rb.
\]
Our claim would therefore follow from the injectivity of the morphism
\[
 g:\varprojlim H^1_\Sigma(\QQ(\mu_{p^n}),\Tfs)^\theta\rightarrow \image\Col_{f,i}^\theta\subset\Lambda.
\]

By \cite[Theorem~12.4(3)]{kato04}, the hypothesis \textbf{(irred)} tells us that the inverse limit $\varprojlim H^1(\QQ(\mu_{p^n}),\T_f)^\theta$ is free of rank $1$ over $\Lambda$. Let $z$ be any generator. Then, $g$ is injective if and only if $g(z)\ne0$. But if $g(z)=0$, then the exact sequence \eqref{eq:PT1} would then give
\[
\image\Col_{f,i}^\theta\hookrightarrow \Sel_i(\Afs/\Qinf)^{\vee,\theta}.
\]
This is impossible since the former is a non-zero sub-module of $\Lambda$, whereas the latter is $\Lambda$-torsion due to \textbf{(tor)}. Hence we are done.
\end{proof}

\begin{corollary}\label{cor:vanishH2}
We have $H^2_\Sigma(\Qinf,\Afs)=0$.
\end{corollary}
\begin{proof}
By Lemma~\ref{lem:vanishH1i} and \eqref{eq:PT2}, $H^2_\Sigma(\Qinf,\Afs)$ injects into $\varinjlim \bigoplus H^2(\QQ(\mu_{p^n})_v,\Afs)$. By local Tate duality, this is isomorphic to $ \varprojlim \bigoplus H^0(\QQ(\mu_{p^n})_v,\Tfs)$. We now show that each of these inverse limits is zero.

Suppose that $\varprojlim H^0(\QQ(\mu_{p^n})_v,\T_{f,s})\ne 0$ for some $v$. Then, $H^0(\QQ(\mu_{p^n})_v,\T_{f,s})\ne 0$ for some $n$. It is a free $\cO_E$-module of either rank $1$ or rank $2$. In either case, for $n$ sufficiently large, $H^0(\QQ(\mu_{p^{n+m}})_v,\T_{f,s})=H^0(\QQ(\mu_{p^n})_{v},\T_{f,s})$ for all $m\ge0$. The corestriction map from $H^0(\QQ(\mu_{p^{n+m+1}})_{v},\T_{f,s})$ to $H^0(\QQ(\mu_{p^{n+m}})_v,\T_{f,s})$ is simply multiplication by $p$. But $\displaystyle\varprojlim_{\times p}\cO_E=0$, hence the claim.
\end{proof}

\begin{corollary}\label{cor:surjloc}
The natural map
\[
H^1_\Sigma(\Qinf,\Afs)\rightarrow \prod_{v|\ell,\ell\in\Sigma}H^1_{/i}(\Qvinf,\Afs)
\]
is surjective for both $i=1,2$.
\end{corollary}
\begin{proof}
This follows on combining \eqref{eq:PT2} and Lemma~\ref{lem:vanishH1i}.
\end{proof}

\begin{corollary}\label{cor:selmerquotient}
We have the isomorphism
\[
\Sel^{\Sigma_0}_i(\Afs/\Qinf)/\Sel_i(\Afs/\Qinf)\cong
\prod_{v|l,l\in\Sigma_0}H^1_{/i}(\Qvinf,\Afs).
\]
\end{corollary}
\begin{proof}
This is a consequence of \eqref{eq:PT2} and Corollary~\ref{cor:surjloc}.
\end{proof}

\begin{remark}\label{rk:mu}
 By \cite[Proposition~2.4]{greenbergvatsal}, $H^1_{/i}(\Kvinf,\Afs)$ is $\Lambda$-cotorsion with zero $\mu$-invariant for all $v\nmid p$. Let $\theta\in\hat\Delta$. Since $H^1_{/i}(\Qvinf,\Afs)^\theta\cong H^1_{/i}(\Kvinf,\Afs(\theta^{-1}))$, the same can be said about the latter.  Therefore, under the hypothesis \textbf{(tor)},  the isomorphism of Corollary~\ref{cor:selmerquotient} tells us that $\Sel^{\Sigma_0}_i(\Afs/\Qinf)^\theta$ is also $\Lambda$-cotorsion. Furthermore, its $\mu$-invariant coincides of that of $\Sel_i(\Afs/\Qinf)^\theta$. (See \S \ref{section:invariants} for the definition of $\mu$-invariants.)
\end{remark}

\subsection{Signed Selmer groups over $\QQ$}\label{section:signed-selmer-over-QQ}
In this section, we study the signed Selmer groups $\Sel_i(\Afs/\QQ)$ and compare them to the Bloch-Kato Selmer group $\Sel_{\mathrm{BK}}(\Afs/\QQ)$. We begin by comparing our local condition at $p$, $H^1_i(\Qp,\Afs)$, with the Bloch-Kato local condition $H^1_f(\Qp,\Afs)$. By definition, this is equivalent to  comparing $H^1_i(\Qp,\Tfs)$ and $H^1_f(\Qp,\Tfs)$.

Recall that $H^1_i(\Qp,\Tfs)$ is defined to be the image of $\ker\Col_{f,i}\otimes\chi^{-s}$ under the natural map $\HIw(\Qp,\Tfs)\rightarrow H^1(\Qp,\Tfs)$. We recall the definition of the Bloch-Kato local condition $H^1_f(\Qp,\Tfs)$ later, but in the next proof we will use the fact that it is the kernel of the Bloch-Kato dual-exponential map $\exp^*$.

\begin{proposition}\label{prop:local-at-p-comparison}
Let $s\in\{0,1,\ldots, k-2\}$, $z\in\HIw(\Qp,\T_f)$ and write  $z_s$ to be the image of $z$ inside $H^1(\Qp,\Tfs)$. Then, $\chi^s(\Col_{f,2}(z))=0$ if and only if $z_s\in H^1_f(\Qp,\Tfs)$.
\end{proposition}
\begin{proof}
By \cite[Theorem~B.5]{loefflerzerbes14}, we have the interpolation formula
\[
\chi^s\left(\cL_{\T_f}(z)\right)=s!(1-p^s\vp)(1-p^{-1-s}\vp^{-1})^{-1}\left(\exp^*(z_s)\otimes t^{-s}e_s\right),
\]
which is equal to
\[
\begin{pmatrix}
\nu_1&\nu_2
\end{pmatrix}
\cdot \chi^s(M)\cdot
\begin{pmatrix}
\chi^s\left(\Col_{f,1}(z)\right)\\ \chi^s( \Col_{f,2}(z))
\end{pmatrix}
\]
thanks to \eqref{eq:decomp}. We recall from \cite[Lemma~3.6]{leiloefflerzerbes15} that $\chi^s(M)$ is equal to the matrix of $\vp$ with respect to the basis $\nu_1$, $\nu_2$. Hence, we deduce that
\begin{equation}
s!\left(\exp^*(z_s)\otimes t^{-s}e_s\right)=
(1-p^s\vp)^{-1}(\vp-p^{-1-s})\begin{pmatrix}
\nu_1&\nu_2
\end{pmatrix}
\cdot
\begin{pmatrix}
\chi^s\left(\Col_{f,1}(z)\right)\\ \chi^s( \Col_{f,2}(z))
\end{pmatrix}.\label{eq:interpolexp}
\end{equation}
Since $\vp(\nu_1)=\nu_2$ and $\vp^2-a_pp^{1-k}\vp+p^{1-k}\epsilon(p)=0$ on $\Dcris(\T_f)$, the right-hand side can be rewritten as
\[
\begin{pmatrix}
\nu_1&\nu_2
\end{pmatrix}
\cdot B_s\cdot
\begin{pmatrix}
\chi^s\left(\Col_{f,1}(z)\right)\\ \chi^s( \Col_{f,2}(z))
\end{pmatrix},
\]
where $B_s$ is the matrix given by
\[
\frac{1}{1-a_pp^{s-k+1}+p^{2s-k+1}\epsilon(p)}
\begin{pmatrix}
-p^{-1-s}-p^{1-k+s}\epsilon(p)+a_pp^{-k}& p^{-k}(1-p)\epsilon(p)\\
1-p^{-1}&-p^{-1-s}-p^{s-k+1}\epsilon(p)+a_pp^{-k+1}
\end{pmatrix}.
\]
For simplicity, we write $B_s=\begin{pmatrix}
a&b\\ c&d
\end{pmatrix}$.

Since $\exp^*(z_s)\otimes t^{-s}e_s$ lies inside $\Fil^0\Dcris(\T_f)$, the coefficient of $\nu_2$ is forced to be zero. That is,
\[
c\times \chi^s\left(\Col_{f,1}(z)\right)+d\times\chi^s\left(\Col_{f,2}(z)\right)=0.
\]
Note that $c\ne0$.
Using this relation, we deduce that the right-hand side of \eqref{eq:interpolexp} is equal to
\[
\left(a\times \chi^s\left(\Col_{f,1}(z)\right)+b\times\chi^s\left(\Col_{f,2}(z)\right)\right)\cdot \nu_1=\frac{bc-ad}{c}\times \chi^s\left(\Col_{f,2}(z)\right)\cdot \nu_1.
\]
Consequently, $\chi^s\left(\Col_{f,2}(z)\right)=0$ if and only if
 $\exp^*(z_s)\otimes t^{-s}e_s=0$. But the kernel of $\exp^*$ is $H^1_f(\QQ,\Tfs)$, hence we are done.
\end{proof}

\begin{remark}\label{rk:swap1and2}
In fact, if the entry $d$ in the matrix $B_s$ is non-zero, then we may replace  $\Col_{f,2}$ by $\Col_{f,1}$ in the statement of Proposition~\ref{prop:local-at-p-comparison}.
\end{remark}

We may now compare our Selmer groups to those of Bloch-Kato. Let us first recall the relevant definitions. If $V$ is a $\Qp$-vector space equipped with a $G_\QQ$-action, then
\[
H^1_f(\QQ_\ell,V) =  \begin{cases}
      H^1_{\mathrm{ur}}(\QQ_\ell,V) & \ell \neq p,\\
      \mathrm{ker}\left( H^1(\QQ_\ell,V) \to H^1(\QQ_\ell,V\otimes\Bcris) \right) & \ell = p,
   \end{cases}
\]
where $\Bcris$ is Fontaine's ring of periods. If $T$ is a $\Zp$-lattice of $V$ stable under $G_\QQ$, then $H^1_f(\QQ_\ell,T)$ (resp. $H^1_f(\QQ_\ell,V/T)$) is defined as the pre-image (resp. image) of $H^1_f(\QQ_\ell,V)$ under the natural inclusion map (resp. projection map). Letting $H^1_{/f}(\cdots)$ denote the quotient $H^1(\cdots)/H^1_{f}(\cdots)$, the Bloch-Kato Selmer group is defined to be
\[
\Sel_{\mathrm{BK}}(A / \QQ) = \mathrm{ker}\left( H^1_\Sigma(\QQ,A) \rightarrow \prod_{\ell \in \Sigma} H^1_{/f}(\QQ_\ell,A) \right).
\]

\begin{proposition}\label{prop:signed-eq-bk}
For $s\in\{0,\ldots, k-2\}$, we have an equality of Selmer groups
\[
\Sel_2(\Afs / \QQ) = \Sel_{\mathrm{BK}} (\Afs / \QQ).
\]
\end{proposition}

\begin{proof}
For $\ell \neq p$ the local conditions $H^1_i(\Afs/\QQ)$ and $H^1_f(\Afs/\QQ)$ are identical, and at $p$ the local conditions are equivalent by Proposition \ref{prop:local-at-p-comparison}, giving the desired result.
\end{proof}

\begin{remark}
By Remark~\ref{rk:swap1and2}, the same is true for $\Sel_1(\Afs / \QQ) $ if $a_p\ne\epsilon(p) p^{s}+p^{k-s-2}$.
\end{remark}

\section{Submodules of finite index}\label{section:submodulesfinite}
The goal of this section is to prove the following theorem, which is a generalization of  \cite[Proposition 4.8]{Gr} (in the ordinary case) and \cite[Theorem~3.14]{kim13} (in the supersingular case).
\begin{theorem}\label{thm:finiteindex}
For $i\in\{1,2\}$ and $\theta\in\hat{\Delta}$, $\Sel_i(A_f/\Qinf)^\theta$  contains no proper $\Lambda$-submodule of finite index.
\end{theorem}

Let us first prove the following generalization of  \cite[Proposition~18]{kim09}. Our proof is based on that of \cite[Proposition~5]{Gr89}.

\begin{proposition}\label{prop:H1sub}
Let $\theta\in\hat\Delta$ be any character and $s\in\ZZ$. The $\Lambda$-module $H^1_\Sigma(\Kinf,A_f(\kappa^s\theta^{-1}))$ has no non-trivial submodule of finite index.
\end{proposition}
\begin{proof}
By the Hochschild-Serre spectral sequence and Corollary~\ref{cor:vanishH2}, we have the isomorphism
\[
H^1(\Gamma_0,H^1_\Sigma(\Kinf,A_f(\kappa^s\theta^{-1})))\cong H^2_\Sigma(\QQ,A_f(\kappa^s\theta^{-1})).
\]
By \cite[Lemma 4.5]{Gr}, $H^2_\Sigma(\QQ,A_f(\kappa^s\theta^{-1}))$ is a divisible group, so this tells us that
\[
H^1_\Sigma(\Qinf,A_f(\kappa^s\theta^{-1}))_{\Gamma_0}
\]
is divisible. Hence we may conclude as in \cite[proof of Proposition~5]{Gr89}.
\end{proof}

The following is a generalization of \cite[Proposition~3.8]{kim13}.

\begin{proposition}\label{prop:surjectivefinite}
Let $\theta\in\hat\Delta$ and $i\in\{1,2\}$. There exists an integer $s$, such that the natural map
\[
H^1_\Sigma(\QQ,A_f(\kappa^s\theta^{-1}))\rightarrow \prod_{v|\ell,\ell\in\Sigma} H^1_{/i}(\QQ_v,A_f(\kappa^s\theta^{-1}))
\]
is surjective.
\end{proposition}
\begin{proof}
Let $s$ be any integer. The Pontryagin dual of $A_f(\kappa^s\theta^{-1})$ is $\T_f(\kappa^{-s}\theta)=\Tfs(\theta\omega^s)$.
Let $\As=\Tfs\otimes E/\cO_E=A_{f,k-s-2}$. Then, the orthogonality of the local conditions as proved in Lemma~\ref{lem:orthfinite} allows us to apply \cite[Proposition~4.13]{Gr}, which tells us that if $\Sel_i(\As(\theta\omega^s)/\QQ)$ is finite, then the cokernel of the map
 \[
 H^1_{\Sigma}(\QQ,A_f(\kappa^s\theta^{-1}))\rightarrow \prod H^1_{/i}(\Qp,A_f(\kappa^s\theta^{-1}))
 \]
 would be isomorphic to $\As(\theta\omega^s)^{G_\QQ}$. But the latter is zero thanks to \textbf{(inv)}. In other words, the surjectivity of the map above would follow from the finiteness of  $\Sel_i(\As(\theta\omega^s)/\QQ)$.

Consider the Selmer group
\[
\Sel_i(\As(\theta\omega^s)/\Kinf)\cong
\Sel_i(A_f(\theta\omega^{k-2})/K_\infty)\otimes\kappa^{k-s-2}.\]
Let $\eta=\theta^{-1}\omega^{2-k}$ (which is independent of $s$), then this can be rewritten as
\[
\Sel_i(A_f/\QQ_\infty)^\eta\otimes\kappa^{k-s-2}
\]
as explained in Remark~\ref{selmer-twists}. Since \textbf{(tor)} says that $\Sel_i(A_f/\QQ_\infty)^\eta$ is $\Lambda$-cotorsion, for all but finitely many $s\in\ZZ$,
\[
\left(\Sel_i(A_f/\QQ_\infty)^\eta\otimes\kappa^{k-s-2}\right)^{\Gamma_0}
\]
is finite. Therefore, in order to show that $\Sel_i(\As(\theta\omega^s)/\QQ)$ is finite, it suffices to find $s$ satisfying this and that
 the kernel and cokernel of the restriction map
 \[
 \Sel_i(\As(\theta\omega^s)/\QQ)\rightarrow \Sel_i(\As(\theta\omega^s)/\Kinf)^{\Gamma_0}
 \]
 are both finite.

 To show this, we consider the commutative diagram
 \[\begin{array}{ccccccc}
 0&\rightarrow&\Sel_i(A/\QQ)&\rightarrow &H^1_\Sigma(\QQ,A)&\xrightarrow{\psi}&\prod H^1_{/i}(\QQ_\ell,A)\\
 &&\downarrow&&\downarrow&&\downarrow\\
 0&\rightarrow&\Sel_i(A/\Kinf)^{\Gamma_0}&\rightarrow &H^1_\Sigma(\Kinf,A)^{\Gamma_0}&\rightarrow&\prod H^1_{/i}(\Kvinf,A)^{\Gamma_0},
 \end{array}
 \]
 where $A:=\As(\theta\omega^s)$.
 The middle vertical map is an isomorphism thanks to \textbf{(inv)} and the inflation-restriction exact sequence. We shall show that the kernel of the third vertical map is finite; then, upon replacing $\prod H^1_{/i}(\QQ_v,A)$ in the diagram with $\mathrm{Im}\ \psi$, the result would then follow from the snake lemma.

Let $\ell$ be a place of $\QQ$, and let $v$ be one of the finitely many places of $\Kinf$ lying over $\ell$. By the inflation-restriction exact sequence, the kernel of the map $H^1(\QQ_\ell,A) \rightarrow H^1(\Kvinf,A)$ is $H^1(\Kvinf/\QQ_\ell,A^{G_{\Kvinf}})$. If $\ell=p$, this group is trivial by \textbf{(inv)}, so we may suppose that $\ell \neq p$. The kernel is clearly trivial if $\ell$ splits completely over $\Kinf$, so assume otherwise. Let $\gamma_v$ be a topological generator for $\mathrm{Gal}(\Kvinf / \QQ_\ell)$ and write $B=A^{G_{\Kvinf}}$, so $H^1(\Kvinf/\QQ_\ell,A^{G_{\Kvinf}}) \simeq B / (\gamma_v -1)B$. We have an exact sequence
\[
0 \rightarrow A^{G_{\QQ_\ell}} \rightarrow B  \xrightarrow{\gamma_v -1} B \rightarrow B / (\gamma_v-1)B \rightarrow 0.
\]
Note that $A^{G_{\QQ_\ell}}$ is finite for all but finitely many $s$. In this case,   we have that $B_{div} \subset (\gamma_v - 1)B$, where $B_{div}$ is the maximal divisible subgroup of $B$. It follows that $B/(\gamma_v-1)B$ is bounded by $B/B_{div}$, which is finite. Hence, this finishes the proof.
\end{proof}


\begin{proposition}\label{prop:surjinv}
Let $\theta\in\hat\Delta$. There exists an integer $s$ such that the natural map
\[
H^1_\Sigma(\Kinf,A_f(\kappa^s\theta^{-1}))^{\Gamma_0}\rightarrow\prod_{v|\ell,\ell\in\Sigma} H^1_{/i}(\Kvinf,A_f(\kappa^s\theta^{-1}))^{\Gamma_0}
\]
is surjective.
\end{proposition}
\begin{proof}
 By Proposition \ref{prop:surjectivefinite} we can pick $s$ such that the map
 \[
H^1_\Sigma(\QQ,A)\rightarrow \prod_v H^1_{/i}(\QQ_v,A)
\]
 is surjective, where $A=A_f(\kappa^s\theta^{-1})$. As $\Gamma_0$ has cohomological dimension $1$, the Hochschild-Serre spectral sequence shows that the map
 \[
H^1_\Sigma(\QQ,A) \rightarrow H^1_\Sigma(\Kinf,A)^{\Gamma_0}
 \]
 is also surjective. Since we have the commutative diagram
\[
\begin{array}{ccc}
H^1_\Sigma(\QQ,A)&\rightarrow&\prod H^1_{/i}(\QQ_\ell,A)\\
\downarrow&&\downarrow\\
H^1_\Sigma(\Kinf,A)^{\Gamma_0}&\rightarrow &\prod H^1_{/i}(\Kvinf,A)^{\Gamma_0}
\end{array}
\]
it suffices to show that for every $\ell$, the natural map
\[
H^1_{/i}(\QQ_\ell,A) \rightarrow \prod_{v \mid \ell} H^1_{/i} (\Kvinf,A)^{\Gamma_0}
\]
is surjective. For $\ell \neq p$, this similarly follows from the fact that $\Gamma_0$ has cohomological dimension 1 (see also the proof of \cite[Lemma 4.7]{Gr}).

We now consider $\ell=p$. By duality, we have the isomorphism
\[H^1_i(\Kpinf,A)=H^1_i(\Kpinf,A_f(\kappa^s\theta^{-1}))\cong\left( \image(\Col_{f,i})^\theta\otimes\kappa^{-s}\right)^{\vee}.\]
Since $\image(\Col_{f,i})^\theta\subset\Lambda$, we have $\left(\image(\Col_{f,i})\otimes\kappa^{-s}\right)^{\Gamma_0}=0$. Consequently,
\begin{equation}\label{eq:coinvariant}
H^1_i(\Kpinf,A)_{\Gamma_0}=0.
\end{equation}
Consider the short exact sequence
\[
0\rightarrow H^1_i(\Kpinf,A)\rightarrow H^1(\Kpinf,A)\rightarrow H^1_{/i}(\Kpinf,A)\rightarrow 0.
\]
We deduce that
\[
H^1_{/i}(\Kpinf,A)^{\Gamma_0}\cong\frac{ H^1(\Kpinf,A)^{\Gamma_0}}{ H^1_i(\Kpinf,A)^{\Gamma_0}}\cong H^1_{/i}(\QQ_p,A),
\]
where the first isomorphism follows from \eqref{eq:coinvariant}, and the second follows from Hochschild-Serre and Remark \ref{rem:restrictedHi}.
\end{proof}

We now prove Theorem~\ref{thm:finiteindex}. We note that it is enough to show that it holds for $\Sel_i(A_f/\Qinf)^\theta\otimes\kappa^s=\Sel_i(A_f(\kappa^s\theta^{-1})/\Kinf)$ for some $s\in\ZZ$. Let $s$ be an integer satisfying the conclusion of Proposition~\ref{prop:surjectivefinite} and write $A=A_f(\kappa^s\theta^{-1})$. Recall from Corollary~\ref{cor:surjloc} that we have the short exact sequence
\[
0\rightarrow \Sel_i(A/\Kinf) \rightarrow H^1_\Sigma(\Kinf,A)\rightarrow \prod H^1_{/i}(\Kvinf,A)\rightarrow 0.
\]
This gives the exact sequence
\[
 H^1_\Sigma(\Kinf,A)^{\Gamma_0}\rightarrow \prod H^1_{/i}(\Kvinf,A)^{\Gamma_0}\rightarrow\Sel_i(A/\Kinf)_{\Gamma_0} \rightarrow H^1_\Sigma(\Kinf,A)_{\Gamma_0}.
\]
Therefore, on combining this with Propositions~\ref{prop:H1sub} and~\ref{prop:surjinv}, we deduce that $\Sel_i(A/\Kinf)_{\Gamma_0}=0$, which concludes our proof.

\section{Algebraic Iwasawa invariants}\label{section:invariants}
Given a finitely generated $\Lambda$-torsion module $M$, we have the pseudo-isomorphism
\[
M\sim \bigoplus_i \Lambda/\vpi^{n_i}\oplus\bigoplus_j\Lambda/F_j^{m_j}
\]
for some integers $n_i,m_j$ and distingushed polynomials $F_j\in\cO_E[X]$. The $\mu$-invariant of $M$ is defined to be $\sum n_i$ whereas its $\lambda$-invariant is defined to be $\sum m_j\times \deg(F_j)$. Let $\theta\in\hat{\Delta}$. Since $\Sel_i(A_f/\Qinf)^\theta$ is $\Lambda$-cotorsion by \textbf{(tor)}, we may define $\lambda$- and $\mu$-invariants for its Pontryagin dual, which we shall refer simply as the $\lambda$- and $\mu$-invariants of $\Sel_i(A_f/\Qinf)^\theta$.
The goal of this section is to generalize results of \cite{kim09}.

\subsection{Mod $\vpi$ Selmer groups}
Consider $A_f[\vpi]=\frac{1}{\vpi}T_f/T_f(1)\subset V_f/T_f(1)$. Recall from hypothesis \textbf{(inv)} that $A_f^{G_{\Qpinf}}=0$.
\begin{lemma}\label{lem:H1torsion}
Let $s$ be any integer. We have the isomorphism $H^1(K,\Afs[\vpi])\cong H^1(K,\Afs)[\vpi]$ if $K$ is a field that is contained in  $\Qpinf$.
\end{lemma}
\begin{proof}
This follows from the long exact sequence induced from
$$0\rightarrow \Afs[\vpi]\rightarrow \Afs\stackrel{\vpi}{\longrightarrow} \Afs\rightarrow0.$$
\end{proof}

Consequently, we may define
\[
H^1_i(\Qpinf,\Afs[\vpi])=H^1_i(\Qpinf,\Afs)[\vpi]\subset H^1(\Qpinf,\Afs[\vpi])
\]
for $i=1,2$. For $v$ a place of $\Qinf$ with $v\nmid p$, we define
\[
H^1_i(\Qvinf,\Afs[\vpi])=H^1_\ur(\Qvinf,A_f[\vpi])\otimes\kappa^s
\]
as before.
This allows us to define the mod $\vpi$ signed Selmer groups
\[
\Sel_i(\Afs[\vpi]/\Qinf)=\ker\left(H^1_\Sigma(\Qinf,\Afs[\vpi])\rightarrow\prod_{v|\ell,\ell\in\Sigma}H^1_{/i}(\Qvinf,\Afs[\vpi])\right)
\]
and
\[
\Sel_i^{\Sigma_0}(\Afs[\vpi]/\Qinf)=\ker\left(H^1_\Sigma(\Qinf,\Afs[\vpi])\rightarrow\prod_{v|\ell,\ell\in\Sigma\setminus\Sigma_0}H^1_{/i}(\Qvinf,\Afs[\vpi])\right)
\]
as before.

\begin{remark}\label{rk:dual}
Note that under the pairing \eqref{eq:pairing},  $H^1_i(\Qpinf,A_f[\vpi])$ is the Pontryagin dual of $\image\Col_{f,i}/\vpi\image\Col_{f,i}$.
\end{remark}

We now prove the following generalization of \cite[Proposition~2.8]{greenbergvatsal} (ordinary case) and \cite[Proposition~10]{kim09} (supersingular case).
\begin{proposition}\label{prop:torsionSel}
For any $\Sigma_0\subset\Sigma$ that contains all primes that divide $N$ but  not $p$ and $\infty$, we have the $\Lambda(\Gamma)$-isomorphism
\[
\Sel_i^{\Sigma_0}(\Afs/\Qinf)[\vpi]\cong\Sel_i^{\Sigma_0}(\Afs[\vpi]/\Qinf).
\]
for $i=1$ or $2$ and any $j\in\ZZ$.
\end{proposition}
\begin{proof}
By Lemma~\ref{lem:H1torsion}, we have the isomorphism
\[
H^1_\Sigma(\Qinf,\Afs[\vpi])\cong H^1_\Sigma(\Qinf,\Afs)[\vpi].
\]
It is therefore enough to check that the local conditions which define the corresponding Selmer groups are compatible.

Let $\ell \in \Sigma\setminus\Sigma_0$ and $v$ a place of $\Qinf$ such that $v|\ell$. If $v=p$, the local conditions are compatible by definition, so suppose $l \neq p$, and let $I_v$ be the corresponding inertia group. Since $v\nmid Np$, we have $A_f^{I_v}=A_f$. Hence the short exact sequence $0\rightarrow A_f[\vpi]\rightarrow A_f\rightarrow A_f\rightarrow 0$ gives rise to
\[
A_f / \vpi A_f \rightarrow H^1(I_v,A_f[\vpi])\rightarrow H^1(I_v,A_f)[\vpi]\rightarrow0.
\]
 Since $A_f$ is divisible, the first term in this sequence is zero, hence we have an isomorphism
 \[
 H^1(I_v,A_f[\vpi])\cong H^1(I_v,A_f)[\vpi].
\]

By Lemma \ref{lem:H1torsion} and the definition of $H^1_i$,  we deduce that
\[
H^1_i(\Qvinf,\Afs)[\vpi]\cong H^1_i(\Qvinf,\Afs[\vpi])
\]
for all places $v$ that divide $\ell \neq p$ for some $\ell\in\Sigma\setminus \Sigma_0$, which concludes the proof.
\end{proof}

If $\theta\in\hat\Delta$ is any character, we may consider the $\theta$-isotypic component of $\Sel_i(\Afs[\varpi]/\Qinf)$ and $\Sel_i^{\Sigma_0}(\Afs[\varpi]/\Qinf)$ as before. These groups can be considered as Selmer groups of $\Afs(\theta^{-1})$ over $\Kinf$.
\subsection{Congruent modular forms}
 Let  $g=\sum a_n(g)q^n$ be a second modular form of weight $k$, level $N'$ with $p\nmid N'$. On enlarging $E$ if necessary, we assume that $a_n(g)\in E$ for all $n$. Similarly, on enlarging $\Sigma$ if necessary, we assume that $\Sigma$ is a set places of $\QQ$ that contains $p$, $\infty$ and the primes that divide $NN'$. Furthermore, we impose the following hypothesis that comes from \cite{bergerlizhu04}.
 \begin{itemize}
 \item[\textbf{(BLZ)}] $\ord_p(a_p(f)),\ord_p(a_p(g))>\lfloor(k-2)/(p-1)\rfloor$.
 \end{itemize}

 By \cite[Theorem~4.1.1]{bergerlizhu04}, this hypothesis implies that
$$\T_f/\vpi\T_f\cong\T_g/\vpi\T_g$$
as $G_{\Qp}$-representations. Indeed, as explained in the proof of \textit{loc. cit.}, we may choose an $\cO_E\otimes\AQp$ basis $\{n_{f,1},n_{f,2}\}$ (resp. $\{n_{g,1},n_{g,2}\}$) of the Wach module of $\NN(\T_f)$ (resp. $\NN(\T_g)$) so that the respective matrices of $\vp$ and $\gamma\in\Gamma$ are congruent modulo $\vpi$. In particular,
\begin{align}
\label{eq:isomod}\NN(\T_f)/\vpi\NN(\T_f)&\rightarrow
 \NN(\T_g)/\vpi\NN(\T_g)\\
\notag x_1 \cdot n_{f,1}+x_2\cdot n_{f,2}&\mapsto x_1\cdot n_{g,1}+x_2 \cdot n_{g,2}
\end{align}
is an isomorphism of $\cO_E\otimes\AQp$-modules. Furthermore, such an isomorphism is $\vp$- and $\Gamma$-equivariant.

As explained in \S\ref{S:signed}, these bases allow us to define Coleman maps $\Col_{f,i}$ (resp. $\Col_{g,i}$) and the signed Selmer groups $\Sel_i(\Afs/\Qinf)$, $\Sel_i(\Ags/\Qinf)$, $\Sel_i(\Afs[\vpi]/\Qinf)$, $\Sel_i(\Ags[\vpi]/\Qinf)$   for $i=1,2$.
\begin{lemma}\label{lem:congcondition}
We have the $\cO_E\lb \Gamma\rb$-isomorphism
\[
H_i^1(\Qpinf,\Afs[\vpi])\cong H_i^1(\Qpinf,\Ags[\vpi])
\]
for both $i=1,2$ and all $s\in\ZZ$.
\end{lemma}
\begin{proof}
By duality and twisting, Remark~\ref{rk:dual} tells us that it suffices to show that
$$
\image \Col_{f,i}/\vpi\image \Col_{f,i}\cong\image \Col_{g,i}/\vpi\image \Col_{g,i}.
$$
But this is immediate from the isomorphism \eqref{eq:isomod}.
\end{proof}

From now on, we  assume the following hypothesis holds for $f$ and $g$.
\begin{itemize}
\item[\textbf{(cong)}]$T_f/\vpi T_f\cong T_g/\vpi T_g$ as $G_\QQ$-representations.
\end{itemize}

\begin{proposition}\label{prop:congSel}
Let $\Sigma_0 = \Sigma \setminus \{p,\infty \}$. The hypothesis \textbf{(Cong)} implies that
\[
\Sel_i^{\Sigma_0}(\Afs[\vpi]/\Qinf)\cong\Sel_i^{\Sigma_0}(\Ags[\vpi]/\Qinf)
\]
as $\cO_E\lb\Gamma\rb$-modules for both $i=1,2$ and any $s\in\ZZ$.
\end{proposition}
\begin{proof}
Note that $T_f/\vpi T_f(1)\cong A_f[\vpi]$ as $G_\QQ$-modules and similarly for $g$. Hence, \textbf{(Cong)} tells us that
\begin{align*}
H^1_\Sigma(\Qinf, \Afs)&\cong H^1_\Sigma(\Qinf,\Ags);\\
H^1_i(\Qvinf, \Afs)&\cong H^1_i(\Qvinf, \Ags)
\end{align*}
for all $v\nmid p$. Hence, together with Lemma~\ref{lem:congcondition}, we deduce that
\[
H^1_i(\Qvinf, \Afs)\cong H^1_i(\Qvinf, \Ags)
\]
for all $v$ and $i=1,2$. The result then follows.
\end{proof}

\begin{theorem}\label{thm:sigma-invariants}
Let $\theta\in\hat{\Delta}$ be any character. The $\mu$-invariant of $\Sel_i(\Afs/\Qinf)^\theta$ vanishes if and only if that of $\Sel_i(\Ags/\Qinf)^\theta$ vanishes. In this case, the $\lambda$-invariants of $\Sel_i^{\Sigma_0}(\Afs/\Qinf)^\theta$ and $\Sel_i^{\Sigma_0}(\Ags/\Qinf)^\theta$ coincide.
\end{theorem}
\begin{proof}
By Remark~\ref{rk:mu}, we may replace the Selmer groups $\Sel_i$ by $\Sel_i^{\Sigma_0}$ in the first statement of the theorem. Now, the $\mu$-invariant of $\Sel_i^{\Sigma_0}(\Afs/\Qinf)^\theta$ vanishes if and only if $\Sel_i^{\Sigma_0}(\Afs/\Qinf)^\theta[\vpi]$  is cofinitely generated over $E$. Therefore, the first statement is a consequence of Propositions~\ref{prop:torsionSel} and~\ref{prop:congSel}.

We now prove the second statement. We assume that the $\mu$-invariant of $\Sel_i(\Afs/\Qinf)^\theta$ vanishes (hence so do those of $\Sel_i^{\Sigma_0}(\Afs/\Qinf)^\theta$ and $\Sel_i^{\Sigma_0}(\Ags/\Qinf)^\theta$). By Theorem~\ref{thm:finiteindex} and \cite[Lemma~12]{kim09}, the $\lambda$-invariant of $\Sel_i^{\Sigma_0}(A_f/\Qinf)^\theta$ (resp. $\Sel_i^{\Sigma_0}(\Ags/\Qinf)^\theta$) are given by the length of the $\cO_E$-modules $\Sel_i^{\Sigma_0}(\Afs/\Qinf)^\theta[\vpi]$ (resp. $\Sel_i^{\Sigma_0}(\Ags/\Kinf)^\theta[\vpi]$). Hence, we are done by Propositions~\ref{prop:torsionSel} and~\ref{prop:congSel}.
\end{proof}

In light of the previous theorem, we now make the following assumption for the rest of the paper.

\begin{itemize}
\item[\textbf{($\boldsymbol{\mu=0}$)}] The $\mu$-invariants of $\Sel_i(A_f/\Kinf)$ and $\Sel_i(A_g/\Kinf)$ vanish.
\end{itemize}

\begin{remark}We note that our earlier hypothesis \textbf{(irred)} is conjectured to imply \textbf{($\boldsymbol{\mu=0}$)} when $k=2$. (See e.g. \cite[Conjecture 7.1]{perrinriou03} in the elliptic curve case.) The situation is less clear for higher weights; see Section \ref{section:example}.\end{remark}

\section{Parity of ranks}\label{section:parity} In this section we compute the $\cO_E$-corank of some cohomology groups before using the results of \S \ref{section:invariants} to obtain our main result.

Throughout this section, we fix  $\Sigma_0 = \Sigma \setminus \{ p, \infty \}$. Recall that by Corollary \ref{cor:selmerquotient} we have an isomorphism
\[
\Sel^{\Sigma_0}_i(\Afs/\Qinf)/\Sel_i(\Afs/\Qinf)\cong
\prod_{v|l,l\in\Sigma_0}H^1_{/i}(\Qvinf,\Afs)
\]
for any $s\in\ZZ$.
Let $\theta\in\hat\Delta$. On taking $\theta$-isotypic components, this becomes
\[
\Sel^{\Sigma_0}_i(\Afs/\Qinf)^\theta/\Sel_i(\Afs/\Qinf)^\theta\cong
\prod_{v|l,l\in\Sigma_0}H^1_{/i}(\Kvinf,\Afs(\theta^{-1})).
\]
By Remark \ref{rk:mu}, $H^1_{/i}(\Kvinf,\Afs(\theta^{-1}))$ is $\Lambda$-cotorsion with zero $\mu$-invariant for all $v$. We write $\tau_{v,f}(s,\theta)$ for its $\lambda$-invariant.

  Let us write
\[
\cH_\ell(\Kinf,\Afs(\theta^{-1})) = \prod_{v\mid \ell} H^1_{/i}(\Kvinf,\Afs(\theta^{-1})).
\]
Then $\cH_\ell(\Kinf,\Afs(\theta^{-1}))$ is again $\Lambda$-cotorsion with zero $\mu$-invariant. Its $\lambda$-invariant is given by $ \delta_{\ell,f}(s,\theta) := \displaystyle\sum_{v \mid \ell} \tau_{v,f}(s,\theta)$.

For a prime $\ell \in \Sigma_0$, let $I_\ell \subset G_{\QQ_\ell}$ denote the inertia subgroup, and let $\mathrm{Frob}_\ell$ denote the corresponding arithmetic Frobenius element in $\mathrm{Gal}(\QQ_\ell^{\mathrm{unr}} / \QQ_\ell)$. Let $k_E$ be the residue field of $E$ and $x\mapsto\tilde{x}$ denote the reduction map modulo $(\varpi)$.
For $s\in\ZZ$, we write $V_{f,s}=V_f(s)$. The following proposition explains how to compute the values $\tau_{v,f}(s,\theta)$.

\begin{proposition}\label{prop:compute-dl}
Let $\ell\in\Sigma_0$ and write $P_{\ell,s,\theta}(X) = \mathrm{det}((1-\mathrm{Frob}_\ell X \vert_{(V_{f,s}(\theta^{-1}))_{I_\ell}})) \in \cO_E[X]$. Let $d_{\ell,f}(s,\theta)$ denote the multiplicity of $X=\tilde{\ell}^{-1}$ as a root of $\tilde{P}_{\ell,s,\theta} \in k_E[X]$. Then for each $v \mid \ell$, we have $\tau_{v,f}(s,\theta) = d_{\ell,f}(s,\theta)$.
\end{proposition}

\begin{proof}
This follows from the same proof as \cite[Proposition 2.4]{greenbergvatsal}.
\end{proof}

\begin{corollary}\label{cor:slodd}
For each $\ell \in \Sigma_0$, we have $\delta_{\ell,f}(s,\theta) \equiv d_{\ell,f}(s,\theta) \mod 2$.
\end{corollary}

\begin{proof}
For any $v \mid \ell$, let $\Gamma_v$ denote the corresponding decomposition subgroup of $\Gamma$; then the number of $v$ above $\ell$ is exactly $r_\ell = [\Gamma_0 : \Gamma_v]$. It follows that
$$\delta_{\ell,f}(s,\theta) = r_\ell\times d_{\ell,f}(s,\theta).$$
 Since $r_\ell$ is a power of $p$, it is necessarily odd, so the result follows.
\end{proof}

We now compute the parity of $d_{\ell,f}(s,\theta)$ for each $\ell \in \Sigma_0$. First we deal with primes which do not divide the level of $f$.

\begin{lemma}\label{lem:lnotdivideN}
If $\ell \nmid N $, then $d_{\ell,f}(s,\theta) =1$ if and only if $a_\ell(f)\equiv  \theta(\ell)\ell^{s+1} + \epsilon(\ell)\theta^{-1}(\ell)\ell^{k-s-2} \mod \vpi$ and $ \epsilon(\ell)^{-1}\theta^2(\ell) \ell^{2s-k+3} \not\equiv 1 \mod \vpi$.
\end{lemma}
\begin{proof}
If $\ell \nmid N$, then $V_f$ is unramified at $\ell$, hence $(V_{f,s}(\theta^{-1}))_{I_\ell}=V_{f,s}(\theta^{-1})$ is two-dimensional, and it is a standard result that $P_{\ell,s,\theta}$ is given by
\begin{equation*}\label{frobcharbad}
P_{\ell,s,\theta}(X) = 1 - a_\ell(f)\theta^{-1}(\ell) \ell^{-s} X + \epsilon(\ell)\theta^{-2}(\ell)\ell^{-2s+k-1} X^2.
\end{equation*}
By Proposition \ref{prop:compute-dl}, $d_{\ell,f}(s,\theta)$ is odd if and only if $\tilde{\ell}^{-1}$ is a simple root of $\tilde{P}_{\ell,s,\theta}(X)$.

Since the product of the roots of $ P_{\ell,s,\theta}(X)$ is $\epsilon(\ell)^{-1}\theta^2(\ell)\ell^{2s-k+1}$, we have
\[
\tilde{\ell}^{-1} \ \text{is a root} \Longleftrightarrow \left({\epsilon(\ell)}^{-1}\theta^2(\ell)\ell\ ^{2s-k+2} \right)\widetilde{ }\ \  \text{is a root}.
\]

Therefore, $\tilde{\ell}^{-1}$ is a root if and only if
\[
a_\ell(f)\equiv  \theta(\ell)\ell^{s+1} + \epsilon(\ell)\theta^{-1}(\ell)\ell^{k-s-2} \mod \vpi.
\]

Hence, if $\tilde{\ell}^{-1}$ is a root, then it is simple if and only if $\ell^{-1} \not\equiv \epsilon(\ell)^{-1}\theta^2(\ell) \ell^{2s-k+2} \mod \vpi$, which proves the lemma.
\end{proof}

Now we consider primes which divide the level of $f$. Let $M$ denote the conductor of the nebentypus $\epsilon$.

\begin{lemma}\label{lem:ldivideN}
Suppose $\ell \mid N$. Then $d_\ell(f,s)=1$ if and only if $a_\ell(f) \equiv \theta(\ell)\ell^{s} \mod \vpi$.

\end{lemma}
\begin{proof}
It is known (see \cite{Li1975}) that when $\ell \mid N$ we have $a_\ell(f) \neq 0$ if and only if one of the following holds:
\begin{itemize}
\item $\ell \mid\mid N$ and $\ell \nmid M$; or
\item $\mathrm{ord}_\ell(M)=\mathrm{ord}_\ell(N)$.
\end{itemize}
By the results of \cite{Vigneras1989}, these are precisely the situations in which the residual mod $\vpi$ representation associated to $(V_{f,s}(\theta^{-1}))_{I_\ell}$ is one-dimensional, with Frobenius polynomial $\tilde{P}_\ell(X)=1-\tilde{a}_\ell(f) \theta^{-1}(\ell)\ell^{-s}( X$; in any other situation the residual $(V_{f,s}(\theta^{-1}))_{I_\ell}=0$. The lemma follows.
\end{proof}

We define $\cS_{f,s}^\theta \subset \Sigma_0$ to be the subset consisting of the primes $\ell$ such that
\begin{itemize}
\item $\ell \nmid N$, $a_\ell(f)\equiv  \theta(\ell)\ell^{s+1} + \epsilon(\ell)\theta^{-1}(\ell)\ell^{k-s-2} \mod \vpi$ and $ \epsilon(\ell)^{-1}\theta^2(\ell) \ell^{2s-k+3} \not\equiv 1 \mod \vpi$; or
\item $\ell \mid N$ and $a_\ell(f) \equiv \theta(\ell)\ell^{s} \mod \vpi$.
\end{itemize}
\begin{remark} For any particular example, this subset is easy to compute using a computer algebra system such as \cite{Magma}. See also Section \ref{section:example}. \end{remark}
The results of this section can be summarized by the following result.
\begin{proposition}\label{prop:delta-parity}
We have the congruence
\[
\sum_{\ell \in \Sigma_0} \delta_{\ell,f}(s,\theta) \equiv \vert \cS_{f,s}^\theta \vert \mod 2.
\]
\end{proposition}
\begin{proof}
This follows immediately from Corollary \ref{cor:slodd}, the definition of $\cS_{f,s}^\theta$, and Lemmas \ref{lem:lnotdivideN} and \ref{lem:ldivideN}.
\end{proof}

We are now prepared to prove the main result of this section, which relates the Selmer-coranks of $f$ and $g$.

\begin{theorem}\label{thm:main}
For any $\theta\in\hat\Delta$, we have the congruence
\[
\mathrm{corank}_{\cO_E}\ \Sel_i(\Afs(\theta^{-1})/\QQ) + \vert \cS_{f,s}^\theta \vert \equiv \mathrm{corank}_{\cO_E}\ \Sel_i(\Ags(\theta^{-1})/\QQ) + \vert \cS_{g,s}^\theta \vert \mod 2.
\]
\end{theorem}

\begin{proof}
Let $\lambda_{f,i}^\theta$ and $\lambda_{g,i}^\theta$ be the $\lambda$-invariants of   $\Sel_i(\Afs/\Qinf)^\theta$ and  $\Sel_i(\Afs/\Qinf)^\theta$ respectively. By Theorem \ref{thm:sigma-invariants}, the $\lambda$-invariants of $\Sel_i^{\Sigma_0}(\Afs/\Qinf)^\theta$ and $\Sel_i^{\Sigma_0}(\Ags/\Qinf)^\theta$ are equal, so by Corollary \ref{cor:selmerquotient} and Proposition \ref{prop:delta-parity} we have the congruence
\begin{equation}\label{eq:cong1}
\lambda_{f,i}^\theta + \vert \cS_{f,s}^\theta \vert \equiv \lambda_{g,i}^\theta + \vert \cS_{g,s}^\theta \vert \mod 2.
\end{equation}
The proof of \cite[Proposition 3.10]{Gr}  shows that
\begin{equation}\label{eq:cong2}
\mathrm{corank}_{\cO_E}\ \Sel_i(\Afs(\theta^{-1})/\QQ) \equiv \lambda_{f,i}^\theta \mod 2,
\end{equation}
and similarly for $g$. We note that, while the proof in {\em loc. cit.} uses the Cassels-Tate pairing to prove that an appropriate-defined Shafarevich-Tate group has square order, for our purposes we must use the generalized pairing of Flach \cite{Flach}. The auto-orthogonality of Lemma \ref{lem:orthfinite} shows that \cite[Theorem 2]{Flach} applies, and the rest of the proof goes through verbatim.

Combining \eqref{eq:cong1} and \eqref{eq:cong2} now gives the desired result.
\end{proof}

Denote by $L_f(s)$ the complex $L$-function associated to the modular form $f$. Recall that $f$ is of even weight $k=2r$, and let $r_{\mathrm{an}}(f) = \mathrm{ord}_{s=r}L_f(s)$  denote the analytic rank of $f$. We have the following generalization of \cite[Theorem 4]{hatleyparity}.

\begin{corollary}\label{cor:parity}
Assume that $f$ and $g$ have trivial nebentypus. Then we have the congruence
\[
r_{\mathrm{an}}(f) + \vert \cS_{f,r-1}^{\mathbf{1}} \vert \equiv r_{\mathrm{an}}(g) + \vert \cS_{g,r-1}^{\mathbf{1}} \vert \mod 2,
\]
where $\mathbf{1}$ denotes the trivial character on $\Delta$.
\end{corollary}

\begin{proof}
Let us first recall the parity result of Nekovar \cite[Theorem B]{nekovar}, which relates the complex $L$-function of $f$ to the Selmer group at the central critical twist:
\[
\mathrm{corank}_{\cO_E}\ \Sel_{\mathrm BK}(A_{f,r-1}/\QQ) \equiv r_{\mathrm{an}}(f) \mod 2.
\]
As shown in \cite[p.100]{Gr},  $$\mathrm{dim}_E\ \Sel_i(V_{f,s} / \QQ) = \mathrm{corank}_{\cO_E}\ \Sel_i(\Afs/\QQ)$$
for all $s\in\ZZ$. Furthermore, Proposition \ref{prop:signed-eq-bk} tells us that the signed Selmer group $\Sel_2(\Afs / \QQ)$ is equal to the Bloch-Kato Selmer group $\Sel_{\mathrm{BK}}(\Afs / \QQ)$ for $s\in\{0,\ldots, k-2\}$.
Therefore, we deduce that
\[
r_{\mathrm{an}}(f)\equiv\mathrm{corank}_{\cO_E}\ \Sel_i(A_{f,r-1}/\QQ)\mod 2
\]
and similarly for $g$.
Now the result follows from Theorem \ref{thm:main}.
\end{proof}

\begin{remark}
The hypotheses that $f$ and $g$ have trivial nebentypus is made solely so that we can use known cases of the parity conjecture \cite[Theorem B]{nekovar}. Since the parity conjecture is expected to hold more generally, our result should also hold for modular forms of nontrivial nebentypus; indeed, there are four pairs of modular forms of level $N \leq 60$ of weight $k=2$, with nontrivial nebentypus, which are congruent mod $p=3$ and satisfy $a_p(f)=0$, and each of these pairs has been computationally shown to satisfy this theorem. Finding congruences between such forms is computationally taxing, hence the small sample size. Once a congruence is known to exist between two modular forms, computing the relevant data to verify Corollary \ref{cor:parity} is fast. See Section \ref{section:example} for further discussion of the hypotheses imposed throughout our paper.
\end{remark}

\section{Unramified base fields and a transition formula}\label{sec:kida}

\subsection{Signed Selmer groups over unramified base fields}\label{sec:urselmer}

Let $K/\Qp$ be a finite unramified extension. We may identify $\Gamma$ with the Galois group $\Gal(K(\mu_{p^\infty})/K)$. If $T$ is a free rank-$d$ $\cO_E$-module equipped with a crystalline continuous $G_{K}$-action, we write $\NN_K(T)$ for its Wach module over $K$. This is now a module over $\cO_E\otimes\AA_{K}^+$, where $\AA_K^+$ is defined to be $\cO_K\lb\pi\rb=\cO_K\otimes\AQp$. Note that if $T$ is equipped with a $G_{\Qp}$-action, then $\NN_K(T)$ and  $\NN(T)$ (the Wach module over $\Qp$) are related by
\begin{equation}
\NN_K(T)=\cO_K\otimes_{\Zp}\NN(T),\quad\NN_K(T)^{G_{\Qpinf}}=\NN(T).\label{eq:basechangewach}
\end{equation}

Define $\HIw(K,T)$ to be the inverse limit $\varprojlim H^1(K(\mu_{p^n}),T)$. The isomorphism \eqref{eq:Herr} generalizes to
\[
\HIw(K,\T_f)\cong \NN_K(\T_f)^{\psi=1}.
\]
If $n_{1}$, $n_{2}$ is an $\cO_E\otimes\AQp$-basis of $\NN(T)$, then \eqref{eq:basechangewach} tells us that it is an $\cO_E\otimes\AA_K^+$-basis of $\NN_K(T)$. Furthermore, we may define the two Coleman maps
\[
\Col_{K,f,i}:\HIw(K,\T_f)\rightarrow \Lambda(\Gamma)\otimes\cO_K,
\]
given by the relation
\[
(1-\vp)(z)=\Col_{K,f,1}(z)\cdot (1+\pi)\vp(n_{1})+\Col_{K,f,2}(z)\cdot (1+\pi)\vp(n_{2}).
\]
Then, $\ker(\Col_{K,f,i})$ allows us to define the local conditions
\[
H^1_i(K(\mu_{p^\infty}),A_f)\subset H^1(K(\mu_{p^\infty}),A_f)
\]
via the Tate duality
\[
\HIw(K,\T_f)\times H^1(K(\mu_{p^\infty}),A_f)\rightarrow E/\cO_E
\]
as before. Consequently, if $F / \QQ$ is a number field where $p$ is unramified, we have the signed Selmer groups
\[
\Sel_i(A_f/F(\mu_{p^\infty}))=\ker\left(H^1(F(\mu_{p^\infty}),A_f)\rightarrow \prod_v H^1_{/i}(F(\mu_{p^\infty})_v,A_f)\right),
\]
where $v$ runs through all places of $F(\mu_{p^\infty})$ and the local conditions outside $p$ are defined using the unramified subgroups.
We may equally define the corresponding Selmer groups for twists of $A_f$, namely, for $s\in\ZZ$,
\[
\Sel_i(\Afs/F(\mu_{p^\infty}))=\ker\left(H^1(F(\mu_{p^\infty}),\Afs)\rightarrow \prod_v H^1_{/i}(F(\mu_{p^\infty})_v,\Afs)\right),
\]
where the local conditions $H^1_{i}(F(\mu_{p^\infty})_v,\Afs)$ are given by $H^1_{i}(F(\mu_{p^\infty})_v,A_f)\otimes\chi^s$. As before, we may take $\Sigma$ to be a finite set of primes containing $p$, $\infty$, and the primes which divide $N$, and then we have
\[
\Sel_i(\Afs/F(\mu_{p^\infty})) \cong \ker\left(H^1_\Sigma(F(\mu_{p^\infty}),\Afs)\rightarrow \prod_{v
\mid \ell, \ell \in \Sigma} H^1_{/i}(F(\mu_{p^\infty})_v,\Afs)\right),
\]
We write $F_\infty$ for the $\Zp$-cyclotomic extension of $F$. If $\theta\in\hat\Delta$, then we may take $\theta$-isotypic components everywhere as before and we have the local conditions
\[
H^1_i(F_\infty,\Afs(\theta^{-1}))=H^1_i(F(\mu_{p^\infty}),\Afs)^\theta
\]
and the Selmer groups
\[
\Sel_i(\Afs(\theta^{-1})/F_\infty)=\Sel_i(\Afs/F(\mu_{p^\infty}))^\theta.
\]

\subsection{Twists of signed Selmer groups and control theorems}

The local conditions at $p$ we defined in \S\ref{sec:urselmer} are compatible in the following sense.

\begin{lemma}\label{lem:compatible-local-condition}
Let $K/\Qp$ be a finite unramified extension, and let $N=\mathrm{Gal}(K(\mu_{p^\infty})/{\Qpinf})$. For any $s\in\ZZ$ and $i\in\{1,2\}$, we have the isomorphism
\[
H^1_i(K(\mu_{p^\infty}),\Afs)^N\cong
H^1_i(\Qpinf,\Afs).
\]
\end{lemma}
\begin{proof}
We may assume that $s=0$ by twisting. By duality, this is equivalent to saying that
\begin{equation}\label{eq:compatiblekernel}
\left(\ker\Col_{K,f,i}\right)_N\cong\ker\Col_{f,i}.
\end{equation}

We remark that since $K/\Qp$ is unramified, the proof of \cite[Lemma~4.4]{lei09} shows that $A_f^{G_{K(\mu_{p^\infty})}}=0$. Hence, the inflation-restriction exact sequence tells us that
\[
H^1(K(\mu_{p^\infty}),\Afs)^N\cong
H^1(\Qpinf,\Afs).
\]
This implies that
\[
\HIw(K,\T_f)_N\cong \HIw(\Qp,\T_f).
\]
Therefore, our claim in \eqref{eq:compatiblekernel} follows from the corresponding compatibility condition of the Coleman maps, which is a consequence of \eqref{eq:basechangewach}.
\end{proof}

From now on, we fix an integer $s$ and $\theta\in\hat\Delta$. In order to ease notation, we shall write $A$ for $\Afs(\theta^{-1})$.

Let $F$ be a number field as in the previous section. We suppose that $F/\QQ$ is Galois and write $G=\mathrm{Gal}(F_\infty/\Kinf)$. We now establish analogues of the results of \cite[\S 2.3]{pollackweston06}, beginning with the following control lemma.

\begin{lemma}\label{lem:control-base-field}
For $i\in\{1,2\}$, the restriction map
\[
\Sel_i(A / \Kinf) \to \Sel_i(A/F_\infty)^G
\]
has finite kernel and cokernel.
\end{lemma}

\begin{proof}
By definition of the signed Selmer groups, we have the commutative diagram
 \[\begin{array}{ccccccc}
 0&\rightarrow&\Sel_i(A / \Kinf)&\rightarrow &H^1_\Sigma(\Kinf,A)&\rightarrow&\prod H^1_{/i}(K_{\infty,\ell},A)\\
 &&\downarrow^{\alpha} &&\downarrow^{\beta} &&\downarrow^{\gamma_v} \\
 0&\rightarrow&\Sel_i(A/F_\infty)^{G}&\rightarrow &H^1_\Sigma(F_\infty,A)^{G}&\rightarrow&\prod H^1_{/i}(F_{\infty,v},A)^{G}.
 \end{array}
 \]
 Lemma~\ref{lem:compatible-local-condition} tells us that $\ker\gamma_v=0$ for all $v|p$.
Since $G$ is finite and $A$ is cofinitely generated, the inflation-restriction exact sequence shows that $\ker \beta$ and $\coker \ \beta$ are finite. Similarly, the same is true for $\ker\gamma_v$ whenever $v \nmid p$. Hence the result follows from the snake lemma.
\end{proof}

Recall our running assumptions \textbf{(tor)} and \textbf{($\boldsymbol{\mu=0}$)}, which say that $\Sel_i(A/ K_\infty)$ is $\Lambda$-cotorsion with zero $\mu$-invariant. The following result tells us that these properties persist over $F$.

\begin{proposition}\label{prop:tor-F}
If $\Sel_i(A / K_\infty)$ is $\Lambda$-cotorsion with $\mu$-invariant zero, then the same is true for $\Sel_i(A / F_\infty)$.
\end{proposition}
\begin{proof}
This follows from Lemma \ref{lem:control-base-field} and Nakayama's lemma for compact local rings. We refer the reader to the proof of \cite[Corollary 3.4]{hachimori-matsuno} for details.
\end{proof}

Now suppose $[F:\QQ]=p^j$ for some $j$, and fix an extension $E' / E$ which contains $\mu_{p^j}$. Let $\cO_{E'}$ be the ring of integers of $E'$, and let $\Lambda' = \cO_{E'}\lb\Gamma_0\rb$ denote the corresponding Iwasawa algebra. By abuse of notation, let $\hat{G}$ be the set of all characters $\psi : G \rightarrow \cO_{E'}^\times$; since $|G|$ is odd, any such character is necessarily even. For each $\psi \in \hat{G}$ we may form the twist
\[
A^\psi = A \otimes_{\cO_E} \cO_{E'}(\psi),
\]
and we regard the corresponding Selmer group $\Sel_i(A^\psi / K_\infty)$ as a $\Lambda'$-module (so its $\lambda$-invariant refers to its $\cO_{E'}$-corank rather than its $\cO_E$-corank). The following is \cite[Proposition 2.6]{pollackweston06}.

\begin{proposition}\label{prop:character-control}
If $G$ is abelian, then there is a natural map
\[
\sum_{\psi \in \hat{G}} \Sel_i(A^\psi/ K_\infty) \rightarrow \Sel_i(A/ F_\infty) \otimes_{\cO_E} \cO_{E'}
\]
with finite kernel and cokernel.
\end{proposition}
\begin{proof}Using Proposition~\ref{prop:tor-F}, the proof of \textit{loc. cit.} carries over verbatim.
\end{proof}

For each $\psi\in\hat G$, we write $\lambda(A^\psi,\Kinf)$ for the $\lambda$-invariant of $\Sel_i(A^\psi / \Kinf)$, and similarly for Selmer groups over $F_\infty$. The following is an immediate corollary.

\begin{corollary}\label{cor:lambda-sum}
If $\Sel_i(A / \Kinf)$ is $\Lambda$-cotorsion with zero $\mu$-invariant, then $\Sel_i(A^\psi / \Kinf)$ is $\Lambda'$-cotorsion with zero $\mu$-invariant for each $\psi \in \hat{G}$. Moreover, if $G$ is abelian, then we also have
\[
\lambda(A, F_\infty) = \sum_{\psi \in \hat{G}} \lambda(A^\psi, \Kinf).
\]
\end{corollary}

\subsection{Kida formula}

We retain the notation from the previous section. If $L$ is a finite extension of the cyclotomic $\Zp$-extension of $\QQ_\ell$, then by restriction we may view any $E$-linear representation $V$ of $G_\QQ$ as a representation of $G_L$, and we set
\[
m_L(V)=\mathrm{dim}_E (V_{I_L})^{G_L} .
\]

\begin{remark}\label{rem:m-description}
In the notation of \S \ref{section:parity}, we have $m_{\Kinf,v}(A) = \tau_{v,f}(s,\theta)$ by \cite[Proposition 2]{Gr89}. A similar description can be made of $m_{F_\infty,v}(A)$.
\end{remark}

Noting that $m_L(V)$ is invariant under extension of scalars, if $L' / L$ is a finite $p$-extension, we enlarge $E$ if necessary so that it contains the $[L':L]^{\mathrm{th}}$ roots of unity, and we define
\[
m(L' / L,V) = \displaystyle\sum_{\psi \in \mathrm{Gal(L'/L)}} m_L(V)-m_L(V(\psi)).
\]

Denote by $R(F_\infty / \Kinf)$ the set of prime-to-$p$ places of $F_\infty$ which are ramified in $F_\infty / \Kinf$. The results of this paper yield the following ``Kida-type formula'', which is a generalization of \cite[Theorem 2.8]{pollackweston06} to the non-ordinary setting.

\begin{theorem}\label{thm:kida}
Let $F/\QQ$ be a finite Galois $p$-extension which is unramified at $p$. If $\Sel_i(A / K_\infty)$ is $\Lambda$-cotorsion with $\mu$-invariant zero, then $\Sel_i(A/ F_\infty)$ is also $\Lambda$-cotorsion with $\mu$-invariant zero. The corresponding $\lambda$-invariants are related by the transition formula
\[
\lambda(A, F_\infty) = p^j \cdot \lambda(A, \Kinf) + \sum_{v \in R(F_\infty / \Kinf)} m(F_{\infty,v} / K_{\infty,\ell}, A)
\]
\end{theorem}

\begin{proof} By \cite[Lemma 2.9]{pollackweston06} we may assume $F_\infty / \Kinf$ is cyclic of degree $p$. In light of Corollary \ref{cor:lambda-sum}, the first part of the theorem is already proved, and we have
\[
\lambda(A, F_\infty) = \sum_{\psi \in \hat{G}} \lambda(A^\psi, \Kinf).
\]
By our choice of $E'$, each $\psi \in \hat{G}$ is trivial modulo a uniformizer $\vpi'$ of $\cO_{E'}$. It follows that we have an isomorphism $A^\psi[\vpi'] \cong A[\vpi]$, and the formula follows from Proposition \ref{prop:congSel} and Remark \ref{rem:m-description}.
\end{proof}

\begin{remark}If $F / \QQ$ is a finite extension which is unramified at $p$, and if $F' / F$ is any finite $p$-extension which is unramified at $p$, then upon replacing the pair $(\QQ,F)$ with the pair $(F,F')$, the results of this section go through in this slightly more general setting.\end{remark}

\section{Hypotheses and Numerical Computations}\label{section:example}

At the beginning of this paper, we assumed $p$ is an odd prime, and we have considered modular forms $f,g$ of level $N$ and even weight $k \geq 2$ whose Fourier coefficients define totally real fields. At various points in the paper, we have added additional assumptions on our modular forms. For the reader's convenience, we begin by recalling these hypotheses, and when applicable we also give some sufficient conditions for each to hold.

\begin{center}
\begin{tabular}{ |p{10ex}|p{50ex}|p{25ex}| }
 \hline
 \textbf{Label} & \textbf{Hypothesis} & \textbf{Sufficient Conditions} \\
 \hline
 \hline
\textbf{(irred)} & The $G_\QQ$-representation $T_f/\vpi T_f$ is irreducible.& See Lemma \ref{lem:irred} below. \\
\hline
 \textbf{(inv)} & For all $m \geq 0$,  $A_f(m)^{G_{\Qpinf}}=0$. & $k \leq p$ \\
 \hline
 \textbf{(tor)}& $\Sel_i(A_{f,s}/\Qinf)^\theta$ is cotorsion over $\Lambda$ & $a_p(f)=0$ or $k \geq 3$ \\
 \hline
 \textbf{(BLZ)}& $\ord_p(a_p(f)),\ord_p(a_p(g))>\lfloor(k-2)/(p-1)\rfloor$ & \\
 \hline
 \textbf{(cong)}& $T_f/\vpi T_f\cong T_g/\vpi T_g$ as $G_\QQ$-representations. &  Sturm's bound\\
 \hline
\textbf{($\boldsymbol{\mu=0}$)}& The $\mu$-invariants of $\Sel_i(A_f/\Kinf)$ and $\Sel_i(A_g/\Kinf)$ vanish. & Conjecturally: $k=2$\\
 \hline
\end{tabular}
\end{center}
\vspace{3ex}

The following lemma is useful when checking whether $\textbf{(irred)}$ is satisfied.
\begin{lemma}\label{lem:irred}
Let $f=\sum a_n(f) q^n$ be a newform of weight $k$ and level $N$. Let $\vpi \mid p$ be a prime such that the $G_\QQ$-representation $T_f/\vpi T_f$ is reducible. If $p \nmid N$ and $p > k$, then
\[
a_\ell(f) \equiv \ell^{k-1} + 1 \mod \vpi\ \text{for all}\ \ell \equiv 1 \mod N.
\]
\end{lemma}
\begin{proof}
See the proof of \cite[Lemma 2.4]{DFG}.
\end{proof}

The explicit nature of the terms in Corollary \ref{cor:parity} make them highly amenable to computation using a computer algebra system such as Magma \cite{Magma}, and it is easy to compute many examples which verify this result. The code used by the authors to verify the following examples can be found at \cite{MagmaCode}.

 It is also interesting to test which of the above hypotheses are really necessary; for instance, computational experiments suggest that the theorem still holds when $k=2$ and $a_p(f) \neq 0$, which suggests that \textbf{(tor)} also holds in this setting.

\begin{remark} From the table, we see that the most interesting hypothesis (and the most difficult to check) is certainly \textbf{($\boldsymbol{\mu=0}$)}. As noted earlier, when $f$ is weight $2$ and nonordinary at $p$, its $\mu$-invariant is always expected to vanish, but the higher-weight cases are more mysterious. See for example \cite[Remark 5.2.3 and 5.2.4]{pollackweston11b}.
\end{remark}

\begin{example}

Let $p=5$ and $k=4$, so \textbf{(inv)} and \textbf{(tor)} are automatically satisfied. Consider the modular forms
\[
f = q - 8q^3 - 10q^5 - 16q^7 + 37q^9 + 40q^{11} + O(q^{12}) \in S_4(\Gamma_0(32)) \quad \text{[LMFDB Label 32.4.1.a]}
\]
and
\[
g = q - 3q^3 + 10q^5 - 4q^7 + 9q^9 + 20q^{11} + O(q^{12}) \in S_4(\Gamma_0(96))\quad \text{[LMFDB Label 96.4.1.c]}
\]
Both have trivial nebentypus and are defined over $\QQ$. It is plain to see that \textbf{(BLZ)} is satisfied, and one easily checks that these forms are congruent mod $5$ using Sturm's bound. Applying Lemma \ref{lem:irred} with $\ell = 97$ shows that \textbf{(irred)} holds.

Since $k=4$, we set $s=\frac{k}{2}-1=1$, and the levels give us $\Sigma_0 = \{ 2,3 \}$. We compute the following data to check Corollary \ref{cor:parity}. Note that, in the notation of Section \ref{section:parity}, $\epsilon = \theta = 1$.

\begin{table}[!ht]
  \centering
  \begin{tabular}{|c|c|c|}
    \cline{2-3}
    \multicolumn{1}{c|}{} & $f$ & $g$ \\ \hline
    $a_2$ & $0$ & $0$ \\ \hline
    $a_2 - 2^{s}$ \rule{0pt}{3ex}    & $3 \not\equiv 0 \mod 5$ & $3 \not\equiv 0 \mod 5$ \\  \hline
    $a_3$ & $-8$ & $-3$\\  \hline
    $a_3 - 3^{s}$\rule{0pt}{3ex}    & N/A & $-6 \not\equiv 0 \mod$ \\ \hline
    $a_3 - 3^{s+1} - 3^{k-s-2}$ \rule{0pt}{3ex}& $-20 \equiv 0 \mod 5$ & N/A \\ \hline
    $3^{2s-k+3} - 1$ \rule{0pt}{3ex}& $2 \not\equiv 0 \mod 5$  & N/A \\ \hline
    $\cS$ & $\{3\}$ &  $\emptyset$ \\ \hline
    Analytic Rank & $1$ & $0$\\ \hline
  \end{tabular}
\end{table}

\vspace{3ex}

Thus we have $|\cS_f|=1, |\cS_g| = 0$, $r_\mathrm{an}(f)=1$, and $r_\mathrm{an}(g)=0$, so
\[
|\cS_f| + r_\mathrm{an}(f) \equiv  |\cS_g| + r_\mathrm{an}(g) \mod 2.
\]
Since $f$ has positive analytic rank while $g$ does not, this example illustrates the role that the ``error terms'' $|\cS_f|$ and $|\cS_g|$ play in the congruence of Corollary \ref{cor:parity}.
\end{example}

We have verified our theorem on many pairs of modular forms; the major difficulty in computing these examples is in generating the pairs of congruent modular forms. Once a pair is in hand, computing the terms in Corollary \ref{cor:parity} is extremely fast. We have yet to find a pair of congruent modular forms which violates our theorem, which can be considered evidence for the vanishing of the $\mu$-invariant for each of these forms.

In the appendix, we present a table of the pairs of congruent modular forms for which we have computationally verified Corollary \ref{cor:parity} using the code \cite{MagmaCode}. Modular forms are listed according to their LMFDB label.

\section*{Acknowledgements} It is a pleasure to thank Rob Harron, Byoung Du Kim, Chan-Ho Kim, and Florian Sprung for their helpful conversations regarding an earlier version of this paper. We are especially grateful to Rob Pollack for pointing out an error in a previous version of this paper. We also thank the referee for helpful comments and suggestions that have been incorporated into the final version of this paper.

\section*{Appendix: Tables of Data}

\begin{table}[!ht]
\caption{Pairs $(f,g)$ for $p=3$, $k=2$}
  \begin{tabular}{cc}

  \begin{tabular}{|c|c|}
\hline
    17.2.1.a & 119.2.1.b  \\ \hline
     & 187.2.1.b \\ \hline
     & 187.2.1.c \\ \hline \hline

     26.2.1.b & 52.2.1.a \\ \hline
     & 182.2.1.c \\ \hline
     & 182.2.1.d \\ \hline \hline

     40.2.1.a & 200.2.1.a	\\ \hline \hline

     46.2.1.a & 92.2.1.a \\ \hline \hline

     52.2.1.a & 182.2.1.c	\\ \hline
     & 182.2.1.e	\\ \hline \hline

     58.2.1.a & 116.2.1.a \\ \hline \hline

     62.2.1.a & 124.2.1.b	\\ \hline \hline

     70.2.1.a & 140.2.1.b	\\ \hline \hline

     71.2.1.b	& 71.2.1.a	\\ \hline
     & 142.2.1.b	\\ \hline
     & 142.2.1.c	 \\ \hline
     & 142.2.1.d	 \\ \hline \hline

	71.2.1.a  & 142.2.1.b	\\ \hline
     & 142.2.1.c	 \\ \hline
     & 142.2.1.d	 \\ \hline

  \end{tabular}
&

  \begin{tabular}{|c|c|}
\hline
         73.2.1.a	& 73.2.1.c \\ \hline \hline

     73.2.1.c & 122.2.1.b	\\ \hline
     & 188.2.1.b	\\ \hline \hline

     77.2.1.a	& 154.2.1.a	\\ \hline
     & 154.2.1.c	 \\ \hline \hline

     94.2.1.a	& 188.2.1.b	 \\ \hline \hline

     119.2.1.b	& 187.2.1.b	\\ \hline
     &  187.2.1.c	\\ \hline \hline

     122.2.1.b	& 188.2.1.b	\\ \hline \hline

     142.2.1.b	& 142.2.1.c	\\ \hline
     & 142.2.1.d	\\ \hline \hline

	 142.2.1.c	& 142.2.1.d	\\ \hline \hline

	 154.2.1.a	& 154.2.1.c	\\ \hline \hline

	 163.2.1.a & 163.2.1.b	\\ \hline \hline

	 182.2.1.c	& 182.2.1.e	\\ \hline \hline

	 184.2.1.c & 184.2.1.d	\\ \hline \hline

	 187.2.1.b	 & 187.2.1.c \\ \hline \hline

	 200.2.1.c	& 200.2.1.e \\ \hline

  \end{tabular}

  \end{tabular}
  \end{table}

  \begin{table}[!ht]
\caption{Pairs $(f,g)$ for $p=5$, $k=2$}
  \begin{tabular}{c}

  \begin{tabular}{|c|c|}
\hline
    14.2.1.a	 & 182.2.1.b	  \\ \hline \hline

   37.2.1.b & 111.2.1.a	 \\ \hline \hline

   38.2.1.a	& 114.2.1.a	 \\ \hline \hline

   71.2.1.a	& 213.2.1.b	\\ \hline \hline

   92.2.1.b & 276.2.1.a	\\ \hline \hline

   197.2.1.a	& 197.2.1.b	\\ \hline \hline

   253.2.1.a	 & 267.2.1.e	\\ \hline \hline

   267.2.1.b	& 267.2.1.e	\\ \hline

  \end{tabular}

  \end{tabular}
  \end{table}

  \begin{table}[!ht]
\caption{Pairs $(f,g)$ for $p=5$, $k=4$}
  \begin{tabular}{c}

  \begin{tabular}{|c|c|}
\hline
    32.4.1.c	& 	96.4.1.c	  \\ \hline \hline

32.4.1.a	& 96.4.1.f	\\ \hline \hline

51.4.1.c	& 102.4.1.d	\\ \hline \hline

51.4.1.b	& 102.4.1.b	\\ \hline \hline

74.4.1.c	& 91.4.1.c	\\ \hline
& 118.4.1.d	 \\ \hline \hline

91.4.1.c	& 118.4.1.d	 \\ \hline \hline

118.4.1.a	&  118.4.1.d	\\ \hline

  \end{tabular}

  \end{tabular}
  \end{table}

\clearpage

\bibliographystyle{amsalpha}
\bibliography{references}
\end{document}